\documentclass[USenglish, cleveref, nolineno]{socg-lipics-v2021}
\usepackage{amsmath}

\usepackage{color}
\usepackage[dvipsnames]{xcolor}
\definecolor{darkred}{rgb}{0.8, 0.0, 0.0}
\hypersetup{
    colorlinks=true,
    urlcolor = {darkred},
    linkcolor = {darkred},
    citecolor = {darkred},
    linkcolor = {darkred},
    linktoc=all
}
\usepackage{kotex}

\colorlet{Color0}{Dandelion}
\colorlet{Color1}{red}
\colorlet{Color2}{Green}
\colorlet{Color3}{BlueViolet}

\usepackage{bbm}
\usepackage{url}
\usepackage{amsthm} 
\usepackage{amssymb}
\usepackage{framed}
\usepackage{graphicx}
\usepackage{tikz}
\usepackage{tikz-cd}
\usetikzlibrary{positioning}
\usetikzlibrary{math}
\usepackage{mathtools}
\usepackage{etoolbox}
\usepackage{extpfeil}
\usepackage{soul}
\usepackage{mathtools}

\usepackage{multicol}

\usetikzlibrary{matrix,arrows,decorations.pathmorphing,decorations.pathreplacing,patterns,shapes.geometric,fadings}

\theoremstyle{remark}
\theoremstyle{definition}
\newtheorem{convention}[theorem]{Convention}

\newtheorem*{question}{Question}
\counterwithin*{claim}{theorem}

\Crefformat{claim}{Claim #2#1#3}
\crefformat{enumi}{(#2#1#3)}

\newcommand{\mapsfrom}{\mathrel{\reflectbox{\ensuremath{\mapsto}}}}
\newcommand{\minzz}{\min_{\mathrm{ZZ}}}
\newcommand{\maxzz}{\max_{\mathrm{ZZ}}}

\newcommand{\Rplus}{\mathbb{R}_{\geq 0}}
\newcommand{\Zplus}{\mathbb{Z}_{\geq 0}}
\newcommand{\Zop}{\mathbb{Z}^{\mathrm{op}}}

\newcommand{\seg}{\mathrm{Seg}}

\newcommand{\cov}{\mathrm{Cov}}

\newcommand{\rank}{\mathrm{rank}}
\newcommand{\Int}{\mathrm{Int}}

\newcommand{\I}{\mathbb{I}}
\newcommand{\id}{\mathrm{id}}
\newcommand{\Pb}{P}

\newcommand{\F}{\mathbb{F}}
\newcommand{\vect}{\mathbf{vec}_{\F}}

\newcommand{\m}{m}
\newcommand{\Mtilde}{\widetilde{M}}

\newcommand{\Fcal}{\mathcal{F}}
\newcommand{\im}{\mathrm{im}}

\newcommand{\Hrm}{\mathrm{H}}
\newcommand{\Z}{\mathbb{Z}}
\newcommand{\R}{\mathbb{R}}

\newcommand{\N}{\mathbb{Z}_{\geq0}}

\newcommand{\dgm}{\mathrm{dgm}}

\newcommand{\abs}[1]{\left\lvert{#1}\right\rvert}
\newcommand{\norm}[1]{\lVert{#1}\rVert}
\newcommand{\Cech}{\check{\mathrm{C}}\mathrm{ech}}
\newcommand{\DCech}{\mathrm{D}\check{\mathrm{C}}\mathrm{ech}}
\newcommand{\Rips}{\mathrm{Rips}}
\newcommand{\DRips}{\mathrm{DRips}}

\newcommand{\dd}{\mathsf{d}}
\newcommand{\rk}{\mathrm{rk}}

\newcommand{\barc}{\mathrm{barc}}

\newcommand{\RNum}[1]{\uppercase\expandafter{\romannumeral #1\relax}}

\newcommand{\tcdot}{\textperiodcentered}

\title{Super-Polynomial Growth of the Generalized Persistence Diagram} 

\author{Donghan Kim}{Department of Mathematical Sciences, KAIST, South Korea}{patrick6231@kaist.ac.kr}{https://orcid.org/0009-0009-1286-5195}{}
\author{Woojin Kim}{Department of Mathematical Sciences, KAIST, South Korea \and \url{https://wj-kim.com/}}{woojin.kim@kaist.ac.kr}{https://orcid.org/0000-0001-8081-5872}{} 
\author{Wonjun Lee}{Department of Mathematical Sciences, KAIST, South Korea}{21june314@kaist.ac.kr}{https://orcid.org/0009-0000-2391-8410}{}

\authorrunning{D. Kim, W. Kim and W. Lee}
\Copyright{Donghan Kim, Woojin Kim and Wonjun Lee}



\funding{This work was supported by the National Research Foundation of Korea(NRF) grant funded by the Korea government(MSIT) (RS-2025-00515946).}

\acknowledgements{WK thanks Jisu Kim and Facundo M\'emoli for their feedback on parts of an earlier draft of this paper. 
}

\begin{document}

\maketitle

\hideLIPIcs

\begin{abstract}
The Generalized Persistence Diagram (GPD) for multi-parameter persistence naturally extends the classical notion of persistence diagram for one-parameter persistence. However, unlike its classical counterpart, computing the GPD remains a significant challenge. The main hurdle is that, while the GPD is defined as the M\"obius inversion of the Generalized Rank Invariant (GRI), computing the GRI is intractable due to the formidable size of its domain, i.e., the set of all connected and convex subsets in a finite grid in $\R^d$ with $d \geq 2$. This computational intractability suggests seeking alternative approaches to computing the GPD.

In order to study the complexity associated to computing the GPD, it is useful to consider its classical one-parameter counterpart, where for a filtration of a simplicial complex with $n$ simplices, its persistence diagram contains at most $n$ points. This observation leads to the question: \emph{Given a $d$-parameter simplicial filtration, could the cardinality of its GPD (specifically, the support of the GPD) also be bounded by a polynomial in the number of simplices in the filtration?} This is the case for $d=1$, where we compute the persistence diagram directly at the simplicial filtration level. If this were also the case for $d\geq2$, it might be possible to compute the GPD directly and much more efficiently without relying on the GRI.

We show that the answer to the question above is negative, demonstrating the inherent difficulty of computing the GPD. More specifically, we construct a sequence of $d$-parameter simplicial filtrations where the cardinalities of their GPDs are not bounded by any polynomial in the number of simplices. Furthermore, we show that several commonly used methods for constructing multi-parameter filtrations can give rise to such ``wild'' filtrations.

\end{abstract}

\section{Introduction}

\subparagraph*{(Multi-parameter) Persistent Homology.} Persistent homology is a central concept in Topological Data Analysis (TDA), which is useful for studying the multi-scale topological features of datasets \cite{carlsson2009topology,carlsson2021topological,dey2022computational}.
In the classical one-parameter setting, topological features in datasets are summarized in the form of persistence diagrams or barcodes \cite{cohen2007stability,zomorodian2005computing}, which serve as complete, discrete invariants of the associated persistence modules. 

Multi-parameter persistent homology generalizes persistent homology. In the $d$-parameter setting, topological features of datasets are indexed by the poset $\R^d$ or a subposet of $\R^d$ \cite{botnan2022introduction,carlsson2009theory}. 
This generalization allows for a finer extraction of topological features, enabling a more detailed analysis of complex datasets 
\cite{carriere2020multiparameter, Corbet2019, Loiseaux2023b, loiseaux2024stable, Scoccola2024, Vipond2020, xin2023gril}.
However, the algebraic structure of multi-parameter persistence modules, i.e. functors from the poset $\R^d$ ($d\geq 2$) to the category $\vect$ of finite dimensional vector spaces over a field $\F$, is significantly more complex, and in fact no complete discrete invariant exists for multi-parameter persistence modules \cite{carlsson2009theory}.
\subparagraph*{Generalized Persistence Diagram (GPD).} Many incomplete, but potentially useful invariants for multi-parameter persistence modules have been studied; e.g. \cite{asashiba2023approximation,asashiba2019approximation,blanchette2021homological,botnan2024signed,chacholski2022effective,gulen2023orthogonal,harrington2019stratifying,landi2018rank,lesnick2015interactive,miller2020homological,oudot2024stability,russoldgraphcode}. The Generalized Persistence Diagram (GPD) is one such invariant and is a natural extension of the persistence diagram for one-parameter persistence \cite{kim2021generalized,patel2018generalized}. 
The GPD is defined as the M\"obius inversion of the Generalized Rank Invariant (GRI)\footnote{The term `generalized persistence diagram' sometimes refers to different concepts; e.g. 
\cite{bubenik2024topological,morozov2021output,patel2024poincare}. Also, we note that the notion of the GRI stems from the concept of the rank invariant \cite{carlsson2009theory}. }, which captures ``persistence'' in multi-parameter persistence modules or, more generally, any linear representations of posets \cite{kim2023persistence}. Many aspects of the GPD and GRI--such as stability, discriminating power, computation, connections to other invariants, and generalizations--have been studied; e.g.  
\cite{asashiba2024interval, botnan2024signed, clause2022discriminating, dey2022computing, dey2024computing, escolar2024barcoding, kim2024extracting, kim2021bettis}. Also, a vectorization method for the (restricted) GRI has recently been proposed and utilized in a machine learning context \cite{mukherjee2024d,xin2023gril}. There are also other works on invariants of multi-parameter persistence modules that are closely related to the GPD and GRI; e.g. 
\cite{asashiba2023approximation, asashiba2019approximation, asashiba2024interval, blanchette2021homological, cai2021elder, gulen2023orthogonal, hiraoka2023refinement}.

\subparagraph*{Challenges in Computing the GPD.}While many properties of the GPD have been clarified in the aforementioned works, computing the GPD, unlike its classical counterpart, remains a significant challenge.
The main hurdle is that, while the GPD is defined as the M\"obius inversion of the GRI, computing the GRI is intractable mainly due to the formidable size of its domain. For instance, the domain of the GRI of a persistence module over a finite grid $G\subset \R^d$ is the set of all \emph{intervals} in $G$ (cf. \Cref{def:interval,def:generalized rank invariant}). The cardinality of the domain is huge relative to  the cardinality $\abs{G}$ of $G$ even when $d=2$ ; cf. \cite[Theorem 31]{asashiba2022interval}. This computational intractability suggests seeking alternative approaches to computing the GPD.

In order to study the complexity associated to computing the GPD, it is useful to consider its classical one-parameter counterpart, where for a filtration of a simplicial complex with $N$ simplices, its persistence diagram contains at most 
$N$ points. 
This observation leads to the question: \emph{Given a $d$-parameter simplicial filtration, could the cardinality of its GPD (more precisely, the support of the GPD) also be bounded by a polynomial in the number of simplices in the filtration?} In what follows, we make this question more precise.

\subparagraph*{Size of the GPD.}
Let $K$ be an abstract simplicial complex, and let $\Delta K$ be the set of subcomplexes of $K$ ordered by inclusion. 
A monotone map $\Fcal:\R^d\rightarrow \Delta K$ such that there exists a $p \in \R^d$ with $\Fcal_p:=\Fcal(p)=K$   is called a \textbf{($d$-parameter simplicial) filtration (of $K$)}. 
The number of simplices in the filtration $\Fcal$ is defined as the number of simplices in $K$.
\begin{definition}[{\cite{lesnick2015interactive}}]\label{def:finite_filtration}We call the filtration $\Fcal$ \textbf{finite}, if $K$ is finite, and in $\Fcal$, every simplex $\sigma\in K$ \emph{is born at finitely many points of} $\R^d$, i.e.  
 there exist finitely many points $q_1,\ldots,q_{t(\sigma)}\in \R^d$ s.t. $\sigma\in \Fcal_p$ for $p\in \R^d$ if and only if $q_i\leq p$ in $\R^d$ for some $i=1,\ldots,t(\sigma)$. We call those points $q_1,\ldots,q_{t(\sigma)}$ the \textbf{birth indices} of $\sigma$. We also say that $\sigma$ is \textbf{born} at $q_1,\ldots,q_{t(\sigma)}\in \R^d$.
\end{definition}
If $\Fcal$ is finite, then for each homology degree $\m\in \Zplus$, the persistence module $\Hrm_m(\Fcal;\F):\R^d\rightarrow \vect$ is finitely presentable (cf. \Cref{def:fp} and Remark \ref{rem:finite_filtration_and_fp}). This guarantees that
the \textbf{$m$-th GPD} of any finite $d$-parameter filtration is well-defined as an integer-valued function on a \emph{finite} subset of intervals of $\R^d$ (cf. \Cref{def:generalized rank invariant} and \Cref{rem:discrete_poset}), denoted by $\dgm_\m(\Fcal)$. By the \textbf{size} of $\dgm_\m(\Fcal)$
we mean the cardinality of the \emph{support} of $\dgm_\m(\Fcal)$, i.e. $|\{I \in \Int(\R^d) : \dgm_\m(\Fcal)(I) \neq 0\}|$ (cf. \Cref{def:interval}).
Since the notion of the GPD reduces to that of the persistence diagram when $d=1$ \cite[Section 3]{clause2022discriminating}, in this case, the size of the GPD becomes the number of points in the persistence diagram. Our question is as follows.

\begin{framed}
\begin{question} 
Does there exist $k\in \N$ such that for \emph{any} homology degree $m\geq 0$, the size of the $\m$-th GPD of \emph{any} finite simplicial filtration $\Fcal$ over $\R^d$ is $O(N^k)$ where $N$ stands for the number of simplicies in $\Fcal$? 
\end{question}
\end{framed}
\vspace{2mm}

Note that when $d=1$, the answer to the question is \emph{yes}, and $k$ can be taken to be $1$, in which case we compute the persistence diagram directly at the simplicial filtration level  \cite{edelsbrunner2002topological}. If the answer is also \emph{yes} when  $d\geq2$, it might be possible to compute the GPD of multi-parameter filtrations directly and much more efficiently without relying on the GRI of the persistence module $\Hrm_m(\Fcal;\F):\R^d\rightarrow \vect$.

\paragraph*{Our contributions.} 

\begin{enumerate}
\item We show that the answer to \textbf{Question} above is negative, demonstrating the inherent difficulty of computing the GPD (\Cref{thm:nonpolynomial size1}). More specifically, we construct a sequence $(\Fcal_n)$ of finite $d$-parameter simplicial filtrations where the sizes of their GPDs are not bounded by any polynomial in the number of simplices in $\Fcal_n$.

\item We also show that several well-known methods for constructing multi-parameter filtrations ---Sublevel-Rips, Sublevel-\v{C}ech, Degree-Rips, and Degree-\v{C}ech---can give rise to such ``wild'' filtrations (\Cref{thm:degree-Rips,thm:sub-Rips}). 

\item We find that computing the value of the GPD 
for a 2-parameter filtration containing $O(n^s)$ simplices, for some $s \in \mathbb{N}$, via the M\"obius inversion of the GRI can result in the problem of summing $\Theta(2^n)$ integers, which is an EXPTIME problem 
(\Cref{cor:exp-time}).
This implies that, even with the fully computed GRI, computing the GPD through M\"obius inversion of the GRI remains computationally challenging. 
\end{enumerate}

It is noteworthy that in the proofs of some of our results,
we employ techniques inspired by some recent works
on the GRI  \cite{botnan2024signed,clause2022discriminating} as well as Rota's Galois connections \cite{gulen2022galois}; see the proof of \Cref{thm:nonpolynomial size1} for the case of $d>2$ and that of \Cref{lem:Interval_projection}.

\subparagraph*{Other related works.} 
Botnan, Oppermann, Oudot investigated the computational complexity of the minimal rank decomposition of the standard rank invariant as part of their results \cite[Remark 5.3]{botnan2024signed}. Also, Botnan, Oppermann, Oudot,  Scoccola found an upper bound 
on the size of the rank \emph{exact} decomposition of a persistence module, which is a polynomial in the size of the (usual) multigraded Betti numbers of the module \cite{botnan2022bottleneck}.  

Clause, Kim, M\'emoli identified a \emph{tame} persistence module $M$ over $\Z^2$ (a concept introduced by Miller \cite{miller2020homological}) whose GRI does not admit a M\"obius inversion \cite[Theorem B]{clause2022discriminating}. Our construction of persistence modules in Section \ref{sec:super} is reminiscent of that of $M$, making it natural to interpret $M$ as a limit object of the types of persistence modules we construct. 

Morozov and Patel elucidated a connection between 1- and 
2-parameter persistence settings in order to find an output-sensitive algorithm for computing the M\"obius inversion of the birth-death function \cite{morozov2021output}. We also note that there are many recent examples of utilizing M\"obius inversion to extract information from (multi-parameter) persistent (co)homology: see e.g.   
\cite{betthauser2022graded, gulen2022galois, gulen2023orthogonal, mccleary2022edit, memoli2022persistent, morozov2021output, oudot2024stability, patel2024poincare, thomas2019invariants}.  

Alonso, Kerber, Skraba showed that many multi-parameter filtration constructions, such as density- or degree-Rips bifiltrations, and across a general category of point samples in Euclidean space, the probability of the homology-induced persistence module decomposing into interval modules approaches zero as the sample size tends to infinity \cite{alonso2024probabilistic}.
 
\subparagraph*{Organization.} In \Cref{sec:prelim}, we review basic terminology and well-known theorems related to persistence modules, as well as the GPD and GRI. In \Cref{sec:super,sec:super-metric-space}, we establish our main results, as mentioned above. Finally, in \Cref{sec:discussion}, we discuss future research directions.

\section{Preliminaries}\label{sec:prelim}

\subparagraph*{Persistence modules and barcodes.} Throughout this paper, $P=(P,\leq)$ is a poset, regarded as the category whose objects are the elements of $P$, and for any pair $p,q\in P$, there exists a unique morphism $p\to q$ if and only if $p\leq q$.
All vector spaces in this paper are over a fixed field $\F$. 
Let $\vect$ denote the category of finite-dimensional vector spaces and linear maps over $\F$. A functor $P \to \vect$ will be referred to as a \textbf{persistence module over $P$} or simply a \textbf{$P$-module}. A morphism between $P$-modules is a natural transformation between the $P$-modules. For any $P$-modules $M$ and $N$, their direct sum $M\oplus N$ is defined pointwisely. A $P$-module $M$ is \textbf{trivial} if $M(p)=0$ for all $p\in P$, and we write $M=0$.  A nontrivial $P$-module $M$ is \textbf{indecomposable} if the assumption $M\cong M'\oplus M''$ for some $P$-modules $M'$ and $M''$ implies that either $M'=0$ or $M''=0$. By Krull-Remak-Schmidt-Azumaya's theorem \cite{azumaya1950corrections,botnan2020decomposition}, any $P$-module is isomorphic to a direct sum of indecomposable $P$-modules, and this decomposition is unique up to isomorphism and permutation of summands.

\begin{definition}\label{def:interval}An \textbf{interval} $I$ of $P$ is a subset $I\subseteq P$ such that: 	
\begin{romanenumerate}
    \item $I$ is nonempty.    
    \item If $p,q\in I$ and $p\leq r\leq q$, then $r\in I$. \label{item:convexity}
    \item $I$ is \textbf{connected}, i.e. for any $p,q\in I$, there is a sequence $p=p_0,
		p_1,\cdots,p_\ell=q$ of elements of $I$ with $p_i$ and $p_{i+1}$ comparable for $0\leq i\leq \ell-1$.\label{item:interval3}
\end{romanenumerate}
By $\Int(P)$ we denote the set of all intervals of $\Pb$.
\end{definition}
For any $p\leq q$ in $P$, the set $[p,q]:=\{r\in P: p\leq r\leq q\}$ is called a \textbf{segment}. The collection of all segments (resp. intervals) in $P$ will be denoted by $\seg(P)$ (rep. $\Int(P)$). It is not difficult to see that $\seg(P)\subset\Int(P)$.

Given any $I\in \Int(P)$, the \textbf{interval module} $\F_I$ is the $P$-module, with
\begin{equation}\label{eq:interval module} \F_I(p) := \begin{cases} \F & \mathrm{if \ } p\in I\\
0 & \mathrm{otherwise.}
\end{cases},
\hspace{10mm}
{\F_I}(p\leq q) := \begin{cases} \id_\F & \mathrm{if \ } p\leq q\in I\\
0 & \mathrm{otherwise}\end{cases}\end{equation}

Every interval module is indecomposable \cite[Proposition 2.2]{botnan2018algebraic}. 
\label{nom:barcode} 
A $P$-module $M$ is {called} \textbf{interval-decomposable} if it is isomorphic to a direct sum of interval modules. The \textbf{barcode} of an interval decomposable $P$-module $M\cong \bigoplus_{j\in J}\F_{I_j}$ is defined as the multiset {$\barc(M):=\{I_j:j\in J\}$.}

For $p\in P$, let $p^{\uparrow}$ (resp. $p^{\downarrow}$) denote the set of points $q\in P$ such that $p\leq q$ (resp. $q\leq p$). Clearly, both $p^\uparrow$ and $p^\downarrow$ belong to $\Int(P)$.
\begin{definition}\label{def:fp}
A $P$-module is {called} \textbf{finitely presentable} if it is isomorphic to the cokernel of a morphism 
$
\bigoplus_{b \in B} \F_{p_b^\uparrow} \to \bigoplus_{a \in A}  \F_{p_a^\uparrow},
$
where $\{p_a:a\in A\}$ and $\{p_b:b\in B\}$ are finite multisets of elements of $P$. 
\end{definition}

\begin{remark}\label{rem:finite_filtration_and_fp} It is well known that, for any finite $d$-parameter simplicial filtration $\Fcal$ (\Cref{def:finite_filtration}), the induced $\R^d$-module $\Hrm_m(\Fcal;\F)$, for each $m\in \Zplus$, is finitely presentable. Namely, $\Hrm_m(\Fcal;\F)$ is isomorphic to the cokernel of a morphism $
\bigoplus_{j \in J} \F_{p_j^\uparrow} \to \bigoplus_{i \in I}  \F_{p_i^\uparrow},
$ where $\{p_j:J\in J\}\cup \{p_i:i\in I\}$ is a subset of the smallest finite $d$-d grid in $\R^d$ that contains all the birth indices of $m$- and $(m+1)$-simplices in $\Fcal$. See, e.g. \cite{lesnick2015interactive}.
\end{remark}

\subparagraph*{Generalized Rank invariant and Generalized Persistence Diagram.} 
Any $P$-module $M$ admits both a \textbf{limit} and a \textbf{colimit} \cite[Chapter V]{mac2013categories}: 
A limit of $M$, denoted by $\varprojlim M$, consists of a vector space $L$ together with a collection of linear maps $\{\pi_p:L\to M(p)\}_{p\in P}$ such that 
\begin{equation}\label{eq:limit}
    M(p \leq q)\circ \pi_p =\pi_q\  \mbox{for every $p\leq q$ in $P$}.
\end{equation}
A colimit of $M$, denoted by $\varinjlim M$, consists of a vector space $C$ together with a collection of linear maps $\{i_p:M(p) \to C\}_{p\in P}$ such that 
\begin{equation}\label{eq:colimit}
    i_q\circ M(p \leq q)=i_p\  \mbox{for every $p\leq q$ in $P$}.
\end{equation}
Both $\varprojlim M$ and $\varinjlim M$ satisfy certain universal properties, making them unique up to isomorphism.

Let us assume that $P$ is connected.
The connectedness of $P$ alongside the equalities given in Equations (\ref{eq:limit}) and (\ref{eq:colimit}) imply that $i_p \circ \pi_p=i_q \circ \pi_q:L\rightarrow C$ for any $p,q\in P$. 
This fact ensures that the \textbf{canonical limit-to-colimit map} \(\psi_M:\varprojlim M\longrightarrow \varinjlim M\) given by $i_p\circ \pi_p$ for any $p\in P$ is well-defined. 
 The \textbf{(generalized) rank} of $M$ is defined to be\footnote{{This construction was considered in the study of quiver representations \cite{kinser2008rank}.}}
    $\rank(M):=\rank(\psi_M)$, which is finite as $\rank(M)=\rank(i_p\circ \pi_p)\leq \dim(M(p))<\infty$ for \emph{any} $p\in P$.
The rank of $M$ is a count of the `persistent features' in $M$ that span the entire indexing poset $P$ \cite{kim2023persistence}.

Now, we refine the rank of a $P$-module, which is a single integer, into an integer-valued function. Options for the domain of the function include $\seg(P)$ or the larger set $\Int(P)$. 

\begin{definition}[{\cite{clause2022discriminating, kim2021generalized}}]
\label{def:generalized rank invariant}
Let $M$ be  a $P$-module.
   \begin{romanenumerate}
       \item The \textbf{generalized rank invariant (GRI)} 
    is the map  $\rk_M:\Int(P)\to \Z_{\geq 0}$ given by $
    I\mapsto \rank(M\vert_I)$
 where $M\vert_I$ is the restriction of $M$ to $I$. 
       \item The \textbf{generalized persistence diagram (GPD) of $M$} is defined as the 
       function $\dgm_M:\Int(P)\rightarrow \Z$ that satisfies the equality\footnote{This equality generalizes \emph{the fundamental lemma of persistent homology}  \cite{edelsbrunner2008computational}.} 
    \begin{equation}\label{eq:rk_in_terms_of_dgm}
    \forall I\in \Int(P),\hspace{3mm} \rk_M(I)=\sum_{\substack{J\in \Int(P)\\ J\supseteq I}} \dgm_M(J),
\end{equation}
where the right-hand side of the equation includes only finitely many nonzero summands. 
\label{item:GPD}
   \end{romanenumerate} 
\end{definition}

We also remark that, in Definition \ref{def:generalized rank invariant}, by replacing all instances of $\Int(P)$ by $\seg(P)$, we obtain the notions of the \emph{rank invariant} and its \emph{signed barcode} 
\cite{botnan2024signed,botnan2022bottleneck,carlsson2009theory}. 
\begin{theorem}[Existence and Uniqueness of the GPD]\label{thm:GPD_existence} Let $M$ be a $P$-module.
\begin{romanenumerate}
    \item If $P$ is finite, then the GPD of $M$ exists. \label{item:finite_GPD}
    \item If $P=\R^d$ ($d=1,2,\ldots$), and $M$ is finitely presentable, then the GPD of $M$ exists. \label{item:GPD_existence_fp}
    Furthermore, $\dgm_M$ is finitely supported, i.e. there exist only finite many $I\in \Int(\R^d)$ such that $\dgm_M(I)\neq 0$. 
    \item In each of the previous two cases, the GPD is unique.\label{item:uniqueness}
\end{romanenumerate}
\end{theorem}
\begin{proof}
    \cref{item:finite_GPD}: If $P$ is finite, then $\Int(P)$ is finite. Thus, the claim directly follows from the M\"obius inversion formula \cite{rota1964foundations}.
    \cref{item:GPD_existence_fp}:  This statement is precisely that of \cite[Theorem C(iii)]{clause2022discriminating}.
    \cref{item:uniqueness}: In the setting of \Cref{item:finite_GPD}, the uniqueness is also a direct consequence of the  M\"obius inversion formula. In the setting of \Cref{item:GPD_existence_fp}, the uniqueness is proved in \cite[Proposition 3.2]{clause2022discriminating}. \qedhere
\end{proof}

\begin{remark}[{\cite[Theorem C (iii)]{clause2022discriminating}}]\label{rem:discrete_poset}
For any finitely presentable $\R^d$-module $M$,  
assume that $M$ is the cokernel of a morphism 
$
\bigoplus_{b \in B} \F_{p_b^\uparrow} \to \bigoplus_{a \in A}  \F_{p_a^\uparrow},
$
where $A$ and $B$ are finite index sets.
Let $\mathcal{C}:=\{p_a:a\in A\}\cup\{p_b:b\in B\}\subset \R^d$.
Consider the finite full subposet  $P:=\prod_{i=1}^d\pi_i(\mathcal{C})\subset \R^d$, where $\pi_i:\R^d\rightarrow \R$ is the canonical projection to the $i$-th coordinate. Then, the GPD of $M$ is directly obtained from the GPD of the restriction $M|_P$, as follows.

We extend $\R^d$ to $\R^d\cup\{-\infty\}$ by declaring that $-\infty\leq x$ for all $x\in \R^d$. Let $\lfloor-\rfloor_P:\R^d\rightarrow P\cup\{-\infty\}$ be the map sending each $x\in \R^d$ to the maximal point $p=:\lfloor x\rfloor_P$ in $P\cup\{-\infty\}$ such that $p\leq x$ in $\R^d\cup\{-\infty\}$. For any $I\subseteq \R^d$, let $\lfloor I \rfloor_P$ denote the set of the points $\lfloor x\rfloor_P$ for $x\in I$. Then, the GPD of $M$, i.e. $\dgm_M:\Int(\R^d)\rightarrow \Z$, is equal to
\[I\mapsto \begin{cases} \dgm_{M|_{P}}(\lfloor I\rfloor_P),&\mbox{if $I= \lfloor-\rfloor_P^{-1}(J)$ for some $J\in \Int(P)$}\\0,&\mbox{otherwise.} \end{cases}\]
This implies that \textbf{there exists a bijection between the supports of $\dgm_M$ and $\dgm_{M|_P}$.} 
\end{remark}

We also remark that if $P$ is a finite poset, the condition given in \Cref{eq:rk_in_terms_of_dgm} holds if and only if
\begin{equation}\label{eq:dgm} \forall I\in \Int(P),\hspace{3mm} \dgm_M(I)=\sum_{\substack{J\in \Int(P)\\ J\supseteq I}} \mu(J,I)\cdot\rk_M(J),\end{equation}
where $\mu$ is the \emph{M\"obius function} of the poset $(\Int(P),\supseteq)$; 
 see \cite[Section 3.7]{stanley2011enumerative} for a general reference on M\"obius inversion, and \cite{kim2023persistence} for a discussion of M\"obius inversion in this specific context.

We define the \textbf{size} of $\rk_M$, denoted by $\norm{\rk_M}$, as the cardinality of its support. Likewise, the size of $\dgm_M$ is defined as the cardinality of its support and is denoted by $\norm{\dgm_M}$.

\begin{remark}\label{rem:basic_properties} For any $P$-module $M$, the following holds.
\begin{romanenumerate}
    \item (Monotonicity) $\rk_M(I)\leq \rk_M(J)$ for any pair $I\supseteq J$ in $\Int(P)$ \cite[Proposition 3.8]{kim2021generalized}.\label{item:monotonicity}

    \item For any $I\in \Int(P)$, if  $\rk_M(I)=0$, then $\dgm_M(I)=0$. Hence, we have $\norm{\dgm_M}\leq \norm{\rk_M}$ {\cite[Remark 8]{kim2023persistence}}.\label{item:supports_of_rk_and_dgm}
    \item If $M$ is interval decomposable and $I\in \Int(P)$, then $\dgm_M(I)$ equals the multiplicity of $I$ in $\barc(M)$ \cite[Theorem 2.10]{dey2022computing}.
\end{romanenumerate} 
\end{remark}

In establishing results in later sections, we will utilize the following well-known concrete formulation for the limit of a $P$-module $M$ (see, for instance, \cite[Appendix E]{kim2021generalized}):

\begin{convention}\phantomsection\label{con:limit_formula} 
The limit of a $P$-module $M$ is the pair
    $(L,(\pi_p)_{p\in P})$ described as:
    \[L:=\left\{ (\ell_p)_{p\in P}\in \prod_{p\in P} M(p): \ \forall p\leq q\in P, M(p \leq q)(\ell_p)=\ell_q\right\}\]
    where for each $p\in P$, the map $\pi_p:L\rightarrow M(p)$ is the canonical projection. 
    Elements of $L$ are called  \textbf{sections} of $M$.  
\end{convention}

\subparagraph*{On 2-d grid lattices}

A \textbf{join} (a.k.a. least upper bound) of $S\subseteq P$ is an element $q_0\in P$ such that (i) $s\leq q_0$, for all $s\in S$, and (ii) for any $q\in P$, if $s\leq q$ for all $s\in S$, then $q_0\leq q$.
 A \textbf{meet} (a.k.a. greatest lower bound)  of $S\subseteq P$ is an element $r_0\in P$ such that (i) $r_0\leq s$ for all $s\in S$, and (ii) for any $r\in P$, if $r\leq s$ for all  $s\in S$, then $r\leq r_0$.
 If a join and a meet of $S$ exist, then they are unique. Hence,  whenever they exist, we refer to them as \emph{the} join (denoted by $\bigvee S$) and \emph{the} meet (denoted by $\bigwedge S$), respectively. When $S$ consists of exactly two points $s_1$ and $s_2$, we also use $s_1 \vee s_2$ and $s_1 \wedge s_2$ instead of $\bigvee S$ and $\bigwedge S$, respectively. 
A \textbf{lattice} is a poset such that for any pair of elements, their meet and join exist. For any $n \in \Zplus$, let $[n]:=\{0 < 1 < \cdots < n\}$. By a \textbf{finite $d$-d grid lattice}, we mean a poset isomorphic to the lattice $[n_1] \times [n_2] \times \cdots \times [n_d]$ for some $n_1, n_2, \cdots, n_d \in \Zplus$.

Let $p$ and $q$ be any two points in a poset $P$. We say that $q$ \textbf{covers} $p$ if $p<q$ and there is no $r\in P$ such that $p<r<q$. By $\cov(p)$, we denote the set of all points $q\in P$ that cover $p$.
\begin{lemma}[{\cite[Theorem 5.3]{asashiba2019approximation}}]\label{lem:J-I} 
Let $\mu$ be the M\"obius function of the poset $(\Int([n]^2),\supseteq)$. Then, for all $J,I\in \Int([n]^2)$ with $J\supseteq I$,
\[
\mu(J, I) =
\begin{cases}
1, & \text{if } I = J, \\
\sum\limits_{\substack{J = \bigwedge S \\ \emptyset \neq S \subseteq \cov(I)}} (-1)^{|S|}, & \text{otherwise},
\end{cases}
\]
where $\abs{S}$ denotes the cardinality of $S$. Note that when $I \neq J$ and there is no nonempty $S \subseteq \cov(I)$ such that $J = \bigwedge S$, the sum above is empty, and thus $\mu(J,I) = 0$. We also remark that $\bigwedge S$ represents the smallest interval in $[n]^2$ that contains all the intervals in $S$.
\end{lemma}

    A subposet \( L \subset P \) (resp. \( U \subset P \)) is called a \textbf{lower (resp. upper) fence} of \( P \) if \( L \) is connected, and for any \( q \in P \), the intersection \( L \cap q^\downarrow \) (resp. \( U \cap q^\uparrow \)) is nonempty and connected {\cite[Definition 3.1]{dey2022computing}}.

\begin{lemma}[{\cite[Proposition 3.2]{dey2022computing}}]\label{lem:fence}
    Let \( L \) and \( U \) be a lower and an upper fence of a connected poset \( P \), respectively. Given any \( P \)-module \( M \), we have
\[
\varprojlim M \cong \varprojlim M|_L \quad \text{and} \quad \varinjlim M \cong \varinjlim M|_U.
\]
\end{lemma}

Let $I$ be any subset of $P$. By \( \min(I) \) and \( \max(I) \), we denote the collections of minimal and maximal elements of \( I \), respectively. A \textbf{zigzag poset} of $n$ points is 
    $\bullet_{1}\leftrightarrow \bullet_{2} \leftrightarrow \ldots \bullet_{n-1} \leftrightarrow \bullet_n$
where $\leftrightarrow$ stands for either $\leq$ or $\geq$.

\begin{definition}[{\cite[Definition 3.5]{dey2022computing}}]
For $I \in \Int([n]^2)$ with \( \min(I) = \{p_0,p_1, \ldots, p_k \} \), and \( \max(I) = \{q_0, q_1, \ldots, q_l \} \), we define the following two zigzag posets 
\begin{align*}
    \minzz(I) :=&\, \{ p_0 < (p_0 \vee p_1) > p_1 < (p_1 \vee p_2) > \cdots < (p_{k-1} \vee p_k) > p_k \} \\
    =&\, \min(I) \cup \{ p_i \vee p_{i+1} : i = 0, \dots, k - 1 \}, \\
    \maxzz(I) :=&\, \{ q_0 > (q_0 \wedge q_1) < q_1 > (q_1 \wedge q_2) < \cdots > (q_{l-1} \wedge q_l) > q_l \} \\
    =&\, \max(I) \cup \{ q_i \wedge q_{i+1} : i = 0, \dots, l - 1 \},
\end{align*}
which are lower and upper fences of \( I \), respectively \cite[Section 3.2]{dey2022computing}.

\end{definition}

\section{Super-Polynomial Growth of the GPD}\label{sec:super}

{The goal of this section is to establish the following theorem.}

\begin{theorem}[Super-polynomial Growth of the GPD]
\label{thm:nonpolynomial size1}
Let $d,\m \in \Z_{\geq 0}$ with $d \geq 2$. There does \emph{not} exist $k\in \N$ such that, for \emph{all} finite filtrations $\Fcal$ over $\R^d$ containing 
$N$ simplices, $\norm{\dgm_\m(\Fcal)}$ belongs to $O(N^k)$.

\end{theorem}

Let $K$ be an abstract simplicial complex and let $\m\in \Z_{\geq0}$.  
Let $C_\m(K;\F)$ and $Z_\m(K;\F)$ be the group of $\m$-chains and that of $\m$-cycles of $K$, respectively.
Let $K^\m$ be the $\m$-skeleton of $K$.

To prove the statement, we construct a sequence $(\Fcal_n)$ of filtrations where  $\Fcal_n$ contains $O(n^t)$ simplices for some $t\in \N$, but $\norm{\dgm_\m(\Fcal_n)}\not\in O(n^k)$ for \emph{any} $k\in \N$.
By Remarks \ref{rem:finite_filtration_and_fp} and \ref{rem:discrete_poset}, it suffices 
to construct filtrations $\Fcal_n$ over the discrete posets $[n]^d$.

\begin{remark}In \cite{lesnick2015interactive}, the \emph{size} of a finite $d$-parameter filtration of a simplicial complex $K$ is defined as the total number of birth indices across all $\sigma \in K$. Even if we adopt this quantity as $N$ in the above theorem, the statement remains true.  
\end{remark}

We also remark that some techniques used in the following proof are similar to those employed in the proof of \cite[Theorem B]{clause2022discriminating}.
We first prove the statement for the case of $d=2$. 

\begin{proof}[Proof of Theorem \ref{thm:nonpolynomial size1} when $d=2$]  We consider the cases of $\m=0$ and $\m\geq 1$ separately. 
\subparagraph*{Case 1 ($\m=0$).} Fix $n \in \N$. Let $K:=K_n$ be the simplicial complex on the vertex set $K^0 = \{z_0,z_1,\ldots, z_n\}$ that consists of the 1-simplices $\{z_i,z_j\}$ ($0\leq i<j\leq n$) and their faces.
We define a filtration $\Fcal:=\Fcal_n$ of $K$ over $[n]^2$ as follows. For $(i,j)\in [n]^2$, let 
(see Figure \ref{fig:filt_sec5}): 
\begin{equation}\label{eq:F}
\Fcal_{(i,j)}:=
\begin{cases}
 \emptyset,&\mbox{if $i+j<n$}
\\ K^0 - \{\{z_i\}\},&\mbox{if $i+j=n$}
\\ K^0,&\mbox{if $i + j > n$ and $(i, j)\neq (n, n)$}
\\ K,&\mbox{if $(i, j)= (n, n)$}.
\end{cases}
\end{equation}
Let $M:=\Hrm_0(\Fcal;\F):[n]^2\rightarrow \vect$. Then, we have:

\begin{figure}[t]
    \centering
    \begin{tikzpicture}[scale=0.4,
                        every pic/.style={scale=0.5,transform shape}]
        \foreach \x in {0, 1, 2, 3}{
            \foreach \y in {0, 1, 2, 3}{
            \tikzmath{ 
                integer \xx, \yy;
                \xx = \x-1; \yy = \y-1;
            if \y==0
                then {
                    {\draw (3*\x, 3*\y) node[minimum size=0.9cm, draw, thick] (\x_\y) {};}; }
                else { 
                    {\draw (3*\x, 3*\y) node[minimum size=0.9cm, draw, thick] (\x_\y) {};}; 
                    {\draw[-{Stealth[scale=0.7]}, thick] (\x_\yy) -- (\x_\y); };
                     };
            if \x!=0 then {
                {\draw[-{Stealth[scale=0.7]}, thick] (\xx_\y) -- (\x_\y); }; 
            };
            if \y==0 then {
                {\node (\x_num) at (3*\x, -1.8) {$\x$};};
            };
            if \x==0 then {
                {\node (\y_num) at (-1.8, 3*\y) {$\y$};};
            };  
            if (\x+\y)>=3 then {
                if (\x==3)&&(\y==3) then {
                    {\draw[very thick] (3*\x-0.6, 3*\y-0.6) -- ++(1.2, 0) -- ++(0, 1.2) -- ++(-1.2, 0) -- ++(0, -1.2) -- ++(1.2, 1.2);};
                    {\draw[very thick] (3*\x-0.6, 3*\y+0.6) -- ++(0.4, -0.4);};
                    {\draw[very thick] (3*\x+0.6, 3*\y-0.6) -- ++(-0.4, 0.4);};
                };
                if not((\x==0)&&(\y==3)) then {
                    {\fill[Color0] (3*\x-0.6, 3*\y-0.6) circle (5pt);};
                };
                if not((\x==1)&&(\y==2)) then {
                    {\fill[Color1] (3*\x+0.6, 3*\y-0.6) circle (5pt);};
                };
                if not((\x==2)&&(\y==1)) then {
                    {\fill[Color2] (3*\x-0.6, 3*\y+0.6) circle (5pt);};
                };
                if not((\x==3)&&(\y==0)) then {
                    {\fill[Color3] (3*\x+0.6, 3*\y+0.6) circle (5pt);};
                };
            };  }
        } }
    \end{tikzpicture}
    \hfill
    \begin{tikzpicture}[every node/.style={inner sep=6pt}, baseline=-1.5cm, scale=0.8]
    \begin{scope}[xshift=0cm]
        \fill[fill=gray!40] (3.5, -0.5) -- (3.5, 3.5) -- (-0.5, 3.5) -- ++(0, -1) -- ++(1, 0) -- ++(0, -1) -- ++(1, 0) -- ++(0, -1) -- ++(1, 0) -- ++(0, -1) -- ++(1, 0) -- cycle;
        \foreach \x in {0, 1, 2, 3}{
            \foreach \y in {0, 1, 2, 3}{
            \tikzmath{ 
                integer \xx, \yy;
                \xx = \x-1; \yy = \y-1;
            if \y==0
                then {
                    {\fill (\x, \y) node (\x_\y) {} circle (2pt);}; }
                else { 
                    {\fill (\x, \y) node (\x_\y) {} circle (2pt);}; 
                    {\draw[-{Stealth[scale=0.7]}, thick] (\x_\yy) -- (\x_\y); };
                     };
            if \x!=0 then {
                {\draw[-{Stealth[scale=0.7]}, thick] (\xx_\y) -- (\x_\y); }; 
            }; 
            if \y==0 then {
                {\node (\x_num) at (\x, -0.8) {$\x$};};
            };
            if \x==0 then {
                {\node (\y_num) at (-0.8, \y) {$\y$};};
            };  }
        } }
        \node[color=gray] at (1.5, 4) {$U$};
    \end{scope}
    \begin{scope}[xshift=5.3cm]
        \fill[fill=cyan!40] (3.5, -0.5) -- ++(0, 2) -- ++(-1, 0) -- ++(0, 1) -- ++(-1, 0) -- ++(0, 1) -- ++(-2, 0) -- ++(0, -1) -- ++(1, 0) -- ++(0, -1) -- ++(1, 0) -- ++(0, -1) -- ++(1, 0) -- ++(0, -1) -- ++(1, 0) -- cycle;
        \foreach \x in {0, 1, 2, 3}{
            \foreach \y in {0, 1, 2, 3}{
            \tikzmath{ 
                integer \xx, \yy;
                \xx = \x-1; \yy = \y-1;
            if \y==0
                then {
                    {\fill (\x, \y) node (\x_\y) {} circle (2pt);}; }
                else { 
                    {\fill (\x, \y) node (\x_\y) {} circle (2pt);}; 
                    {\draw[-{Stealth[scale=0.7]}, thick] (\x_\yy) -- (\x_\y); };
                     };
            if \x!=0 then {
                {\draw[-{Stealth[scale=0.7]}, thick] (\xx_\y) -- (\x_\y); }; 
            }; 
            if \y==0 then {
                {\node (\x_num) at (\x, -0.8) {$\x$};};
            };
            if \x==0 then {
                {\node (\y_num) at (-0.8, \y) {$\y$};};
            };  }
        } }
        \node[color=cyan] at (1.5, 4) {$D$};
    \end{scope}
    \end{tikzpicture}
    \caption{For $n=3$, the filtration over $[n]^2$ defined in Equation (\ref{eq:F}), along with the intervals $U$ and $D$ of $[n]^2$, as described in Equation (\ref{eq:U_and_D}).} 
    \label{fig:filt_sec5}
\end{figure}
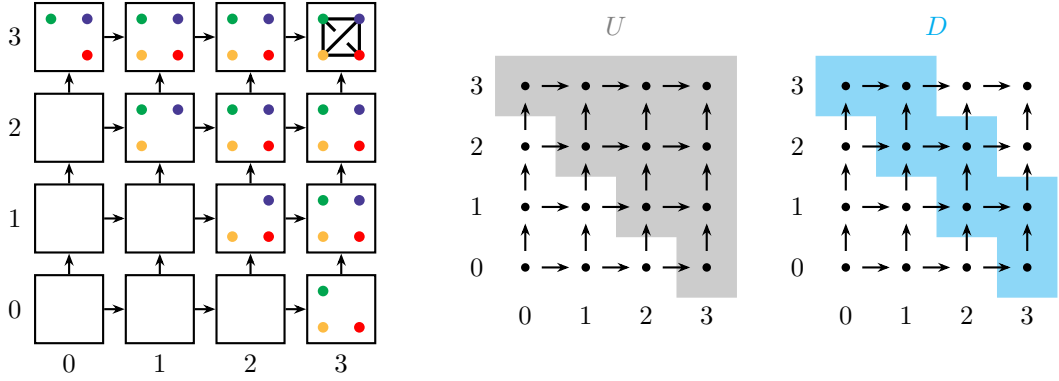
\begin{equation}\label{eq:M}
M{(i, j)}\cong
\begin{cases}
0,&\mbox{if $i+j<n$}
\\\F^n,&\mbox{if $i+j=n$}
\\\F^{n+1},&\mbox{if $i+j>n$ and $(i, j)\neq (n, n)$}
\\\F,&\mbox{if $(i,j)=(n,n)$.}
\end{cases}
\end{equation}
Let $i\in [n]$. Then, for every $(x,y)\in [n]^2$ that covers $(i,n-i)$, 
we have that
\begin{align}
M((i,n-i) \leq (x, y))&&=&&\iota_i^n:&&\F^{n}&\to \F^{n+1}\nonumber
\\
&&&&&&(v_1, \dots, v_n) &\mapsto (v_1, \dots, v_{i}, 0, v_{i+1}, \dots, v_n). \label{eq:iota}
\end{align}
In what follows, we use $\iota_i$ in place of $\iota_i^n$. Now, we consider the following two intervals of $[n]^2$ (see Figure \ref{fig:filt_sec5}):
\begin{equation}\label{eq:U_and_D}
U:=\{(i,j)\in [n]^2 : \ n\leq i + j \}\ \  \mbox{and} \ \ D:=\{(i,j)\in [n]^2 : n\leq i + j \leq n+1 \}.
\end{equation}

By proving the following claims, we complete the proof for the case of $\m=0$.
\begin{claim} \label{cl:rank_J=0}
    If $J\in \Int([n]^2)$ is not a proper subset of $U$, then $\rk_M(J)=0$.
\end{claim}

\begin{claimproof}

First, assume that 
there exists a point $p \in J-U$. Then, by the construction of $M$, $\dim (M(p)) = 0$ and thus $\rk_M (J)=0$ by monotonicity of $\rk_M$ (Remark \ref{rem:basic_properties}~\cref{item:monotonicity}).
Next, assume that $J=U$.
Since $\rk_M(U)=\rank(\varprojlim M|_U\rightarrow \varinjlim M|_U)$, it suffices to show $\varprojlim M|_U=0$. Again, since $D=\minzz(U)$
, by Lemma \ref{lem:fence}, it suffices to show that $\varprojlim M|_D=0$. 
By Equation (\ref{eq:iota}), the diagram $M|_D$ is given by 
\[\F^n\stackrel{\iota_0}{\longrightarrow}\F^{n+1}\stackrel{\iota_1}{\longleftarrow}\F^n\stackrel{\iota_1}{\longrightarrow}\F^{n+1}\stackrel{\iota_2}{\longleftarrow}\F^n\stackrel{\iota_2}{\longrightarrow}\F^{n+1}\stackrel{\iota_3}{\longleftarrow} \F^n \stackrel{\iota_3}{\longrightarrow}\cdots \stackrel{\iota_n}{\longleftarrow}\F^n.\]
From this, it is not difficult to see that $\varprojlim M|_D=0$. 
Without loss of generality, assume that $n=2$. Then, $M|_D$ is given by
$\F^2\stackrel{\iota_0}{\longrightarrow}\F^{3}\stackrel{\iota_1}{\longleftarrow}\F^2\stackrel{\iota_1}{\longrightarrow}\F^{3}\stackrel{\iota_2}{\longleftarrow}\F^2.$
By Convention \ref{con:limit_formula}, {a} vector $v:= (v_1,\ldots,v_{12})\in \F^{12}$ belongs to
$\varprojlim M|_D$ if and only if
\[(v_1,v_2)\stackrel{\iota_0}{\mapsto}(v_3,v_4,v_5)\stackrel{\iota_1}{\mapsfrom}(v_6,v_7)\stackrel{\iota_1}{\mapsto}(v_8,v_9,v_{10})\stackrel{\iota_2}{\mapsfrom}(v_{11},v_{12}),\]
{i.e.,} 
$v_k=0$ for all $k=1,\ldots, 12$.  
 
A more general proof follows. For $p \in D$, let $\pi_p: \varprojlim {M|}_D \to M(p)$ be the canonical projection. Since $M(i, n-i) = \mathbb{F}^n$ for $i = 0, \dots, n$, 
for any $v\in \varprojlim {M|}_D$, let
$\pi_{(i, n-i)}(v) =: (v_1^i, v_2^i, \dots, v_n^i) \in \mathbb{F}^n.$
Equations (\ref{eq:limit}) and (\ref{eq:iota}) imply the following equalities:
\[\begin{array}{rcccl}
    \pi_{(1, n)}(v) &=& \iota_{0} \circ \pi_{(0, n)}(v) &=& (0, v_1^0, v_2^0, \dots, v_n^0), \\
    \pi_{(1, n)}(v) &=& \iota_{1} \circ \pi_{(1, n-1)}(v) &=& (v_1^1, 0, v_2^1, \dots, v_n^1), \\
    \pi_{(2, n-1)}(v) &=& \iota_{1} \circ \pi_{(1, n-1)}(v) &=& (v_1^1, 0, v_2^1, v_3^1, \dots, v_n^1), \\
    \pi_{(2, n-1)}(v) &=& \iota_{2} \circ \pi_{(2, n-2)}(v) &=& (v_1^2, v_2^2, 0, v_3^2, \dots, v_n^2), \\
    && \vdots && \\
    \pi_{(i, n-i+1)}(v) &=& \iota_{i-1} \circ \pi_{(i-1, n-i+1)}(v) &=& (v_1^{i-1}, v_2^{i-1}, \dots, v_{i-1}^{i-1}, 0, v_i^{i-1}, \dots, v_n^{i-1}), \\
    \pi_{(i, n-i+1)}(v) &=& \iota_{i} \circ \pi_{(i, n-i)}(v) &=& (v_1^i, v_2^i, \dots, v_{i-1}^i, v_i^i, 0, \dots, v_n^i), \\
    && \vdots && \\
    \pi_{(n, 1)}(v) &=& \iota_{n-1} \circ \pi_{(n-1, 1)}(v) &=& (v_1^{n-1}, v_2^{n-1}, \dots, v_{n-1}^{n-1}, 0, v_n^{n-1}), \\
    \pi_{(n, 1)}(v) &=& \iota_{n} \circ \pi_{(n, 0)}(v) &=& (v_1^n, v_2^n, \dots, v_{n-1}^n, v_n^n, 0).
\end{array}\]
By comparing 
the $(2i-1)$-th and $(2i)$-th lines for $i=1\ldots,n$, we deduce that $v^i_j = 0$ for all $i = 0, 1, \dots, n$ and $j = 1, 2, \dots, n$. 
Therefore, 
$\pi_{(i, n-i)}(v) = 0$
for all $i=0,1,\ldots,n$ and in turn $v = 0$.
Since $v$ was an arbitrary element of $\varprojlim M|_D$, we have 
$\varprojlim M|_{D}=0$. \claimqedhere

\end{claimproof}

Consider the following intervals of $[n]^2$ (see Figure \ref{fig:Interval_U_i}): 
\begin{equation}\label{eq:Interval_U_i}
    U_i := U-\{(i,n-i)\} \ \ \mbox{for each } i\in [n].
\end{equation}
\begin{figure}[ht]
    \centering
    \begin{tikzpicture}[every node/.style={inner sep=4pt}, scale=0.6]
    \foreach \index in {0, 1, 2, 3}{
    \begin{scope}[xshift=\index*5.5cm]    
        \fill[fill=Color\index!40] (3.5, -0.5) -- (3.5, 3.5) -- (-0.5, 3.5) -- ++(0, -1) -- ++(1, 0) -- ++(0, -1) -- ++(1, 0) -- ++(0, -1) -- ++(1, 0) -- ++(0, -1) -- ++(1, 0) -- cycle;
        \fill[white] (\index+0.5, 3.5-\index) rectangle +(-1.1, -1,1);
        \foreach \x/\xx in {0, 1, 2, 3}{
            \foreach \y/\yy in {0, 1, 2, 3}{
            \tikzmath{ 
                integer \xx, \yy;
                \xx = \x-1; \yy = \y-1;
            if \y==0
                then {
                    {\fill (\x, \y) node (\x_\y) {} circle (2pt);}; }
                else { 
                    {\fill (\x, \y) node (\x_\y) {} circle (2pt);}; 
                    {\draw[-{Stealth[scale=0.7]}, thick] (\x_\yy) -- (\x_\y); };
                     };
            if \x!=0 then {
                {\draw[-{Stealth[scale=0.7]}, thick] (\xx_\y) -- (\x_\y); }; 
            }; 
            if \y==0 then {
                {\node (\x_num) at (\x, -1) {$\x$};};
            };
            if \x==0 then {
                {\node (\y_num) at (-1, \y) {$\y$};};
            };  }
        } }
        \node[color=Color\index] at (1.5, 4) {$U_{\index}$};
    \end{scope}
    }
    \end{tikzpicture}
    \caption{For $n=3$, the intervals $U_0, U_1, U_2, U_3$ of $[n]^2$ that are defined in Equation (\ref{eq:Interval_U_i}).}
    \label{fig:Interval_U_i}
\end{figure}
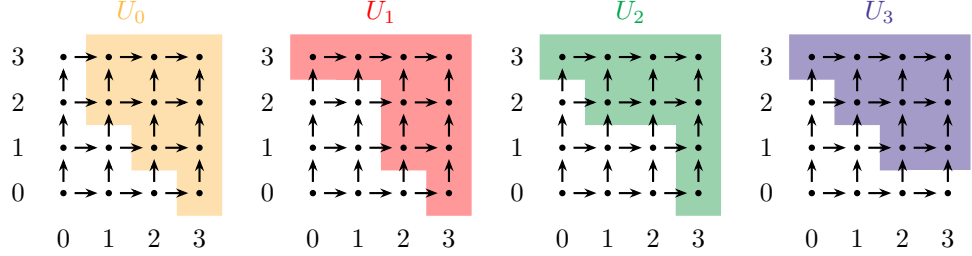

\begin{claim} \label{cl:rank_Ui=1} 
    Let $J\in \Int([n]^2)$ such that $\bigcap_{i=0}^{n}U_i \subseteq J \subsetneq U$. Then, $\rk_M(J)=1$.
\end{claim}
\begin{claimproof}
Since $J \subsetneq U$, there exists $i\in [n]$ such that $J \subseteq U_i$.
Fix such $i$. Since $\max(J)=\{(n,n)\}$, by Lemma \ref{lem:fence}, we have the isomorphism 
$\varinjlim M|_{J} \stackrel{(\ast)}{\cong} M(n,n)$.
Observe that for every $p\in J \subseteq U_i$, $\{z_i\}\in \Fcal_{p}$. Hence, by Convention \ref{con:limit_formula} and the definition of $M$, we have that \[( [z_i] )_{p\in J}\in \varprojlim M|_{J} \subset \bigoplus_{p\in J} \Hrm_0(\Fcal_p;\F).\] 
 Clearly, this tuple $( [z_i] )_{p \in J}$ maps to $[z_i] \in \Hrm_0(\Fcal_{(n,n)};\F) = M(n,n)$ via the limit-to-colimit map of $M$ over $J$, followed by the isomorphism $(\ast)$. Since $[z_i]\in M(n,n)$ is nonzero and $\dim M(n,n)=1$ (cf. \Cref{eq:M}), we have that $\rk_M(J)= 1$, as desired. \claimqedhere
\end{claimproof}

\begin{claim} \label{cl:dgm_Ui=1}
    $\dgm_M(U_i)=1$ for every $i\in [n]$.
\end{claim} 
\begin{claimproof}
We have:
\begin{align*}
1&= \rk_M(U_i)&\mbox{by Claim \ref{cl:rank_Ui=1}}
\\&=\sum_{\substack{I\in \Int([n]^2)\\I\supseteq U_i}}\dgm_M(I)
&\mbox{by Equation \eqref{eq:rk_in_terms_of_dgm}} \\&= \dgm_M(U_i)&\mbox{by Claim \ref{cl:rank_J=0} and Remark \ref{rem:basic_properties}~\cref{item:supports_of_rk_and_dgm}. 
\claimqedhere} 
\end{align*}
\end{claimproof}
For each nonempty $S \subset [n]$, we define $U_{S} := \bigcap_{i \in S} U_i$, which is an interval of $[n]^2$.

\begin{claim} \label{cl:dgm_wedgeS}
For each nonempty $S \subset [n]$, we have that $\dgm_{M}(U_{S}) = (-1)^{|S|+1}$. 
\end{claim}
\begin{claimproof}
Let $I:=U_{S}$.
When $|S|=1$, there exists $i\in [n]$ such that $I=U_i$. We already proved that $\dgm_M(U_i)=(-1)^{1+1}=1$.
If $|S| > 1$, then
\begin{align*}
\dgm_M(I)&=\sum_{\substack{J\in \Int(P)\\ J \supseteq I}}\ \mu(J, I)\cdot  \rk_M(J) &\mbox{by \Cref{eq:dgm}}\\&=\sum_{\substack{J\in \Int(P)\\ U \supsetneq J \supseteq I}}\ \mu(J, I)\cdot  \rk_M(J) 
&\mbox{by Claim \ref{cl:rank_J=0}}
\\ &= \sum_{\substack{J\in \Int(P)\\ U \supsetneq J \supseteq I}} (-1)^{|J-I|}\cdot  \rk_M(J) &\mbox{by Lemma \ref{lem:J-I}}
\\ &= \sum_{S' \subsetneq S} (-1)^{|S'|} 
&\mbox{by Claim \ref{cl:rank_Ui=1}, letting $S' = J - I$} 
\\ &= (1-1)^{|S|} - (-1)^{|S|} = (-1)^{|S|+1}& \mbox{by the binomial theorem}
\end{align*}
as desired. 
The number of nonempty subsets $S\subset [n]$ is $2^{n+1}-1$
and each $S$ corresponds to the interval $U_S$ such that $\dgm_M (U_S) \ne 0$.
Hence, $\norm{\dgm_0(\Fcal)}=\norm{\dgm_M}\geq 2^{n+1}-1.$ 
In particular, there is no $k\in \N$ such that  $2^{n+1}-1\in O(N^k)$ for $N:=(n^2+n)/2$, the number of simplices in the filtration $\Fcal$, completing the proof. 
\end{claimproof}

\subparagraph*{Case 2 ($\m\geq 1$).}
Let $K':=K_{n,m}$ be the smallest simplicial complex 
on the vertex set $\{x_0, \cdots, x_n, y_0, \cdots, y_\m\}$, containing all  $\m$-simplices $\tau, \tau_{ij}$, and  $(\m+1)$-simplices $\sigma_{ij}$ that are the underlying sets of the following oriented simplices (see \Cref{fig:nonpoly_dgm_case2_simpcomp}):
\begin{align*}
    \tau^{+} &:= [y_0,\cdots,y_\m] &  \\
    \tau_{ij}^{+} &:= [x_i, y_0, \cdots, y_{j-1}, y_{j+1}, \cdots, y_\m], &\text{ for } 0 \leq i \leq n,\, 0 \leq j \leq \m, \\
    \sigma_{ij}^{+} &:= [x_{i-1}, x_i, y_0, \cdots, y_{j-1}, y_{j+1}, \cdots, y_\m], &\text{ for } 1 \leq i \leq n,\, 0 \leq j \leq \m.
\end{align*}
For $0 \leq k \leq n$, let $L_k$ be the $\m$-dimensional subcomplex of $K'$, whose $\m$-simplices are exactly $\tau$ and $\{\tau_{kj} : 0 \leq j \leq \m\}$  (see Figure \ref{fig:nonpoly_dgm_case2_simpcomp}). 

We define 
the filtration $\Fcal':=\Fcal'_{n,m}:[n]^2\rightarrow \Delta K'$ ($=\Delta K_{n,m}$) as (see \Cref{fig:nonpoly_dgm_case2_filt}): 
\begin{equation} \label{eq:F'}
    \Fcal'_{(i,j)}:=
    \begin{cases}
        \emptyset,  &\mbox{if $i+j<n$} \\
        \bigcup_{k \in [n]- \{i\}} L_k,    &\mbox{if $i+j=n$} \\
        \bigcup_{k\in [n]} L_k,  &\mbox{if $i + j > n$ and $(i, j)\neq (n, n)$} \\
        K',  &\mbox{if $(i, j)= (n, n)$}.
    \end{cases}
\end{equation}

Note that, for each $0 \leq i \leq n$, the $\m$-chain $c_i = \tau^{+} - \sum_{j=0}^{\m} (-1)^{j}\tau_{ij}^{+}$ in $C_{\m}(K';\F)$ satisfies $\partial c_i = 0$, i.e.
$c_i$ is an $\m$-cycle. Moreover, 
$\Hrm_\m(L_i;\F)$, which is isomorphic to $Z_\m(L_i;\F)$, is the 1-dimensional vector space generated by the cycle $[c_i]$. 

\tikzset{
L0/.pic={
    \draw[Color0, very thick] (-1, -1) -- (-0.5, 0) -- (-1, 1) -- cycle;
    }}
\tikzset{
L1/.pic={
    \draw[Color1, very thick] (-1, -1) -- (0, 0) -- (-1, 1) -- cycle;
    }}
\tikzset{
L2/.pic={
    \draw[Color2, very thick] (-1, -1) -- (0.5, 0) -- (-1, 1) -- cycle;
    }}
\tikzset{
L3/.pic={
    \draw[Color3, very thick] (-1, -1) -- (1, 0) -- (-1, 1) -- cycle;
    }}
\begin{figure}[ht!]
    \centering
    \begin{subfigure}[h]{0.9\linewidth}
        \centering
        \begin{tikzpicture}
            \fill[gray!30] (0, -1.5) -- (1, 0) -- (0, 1.5) -- (7, 0) -- (0, -1.5);
            \draw[thick] (0, -1.5) -- (0, 1.5) -- (1, 0) -- (0, -1.5) -- (2, 0) -- (0, 1.5);
            \draw[thick] (0, 1.5) -- (5, 0) -- (0, -1.5) -- (7, 0) -- (0, 1.5);
            \draw[thick] (1, 0) -- (2.7, 0) (4, 0) -- (7, 0);
            \draw[red, thick] (0.0325, -1.435) -- (0.0325, 1.435) -- (1.95, 0) -- cycle;
            \node at (3.2, 0) {$\cdots$};
            \node[rotate=15] at (3.4, -1.2) {$\cdots$};
            \node[rotate=-15] at (3.7, 1.15) {$\cdots$};
            \node (y0) at (-0.4, 1.6) {$y_0$};
            \node (y1) at (-0.4, -1.6) {$y_1$};

            \node (tau01) at (1, 1.8) {$\tau_{01}$};
            \node (x0) at (1.7, 1.65) {$x_0$};
            \node (tau11) at (2.4, 1.5) {$\tau_{11}$};
            \node (x1) at (3.1, 1.35) {$x_1$};
            \node[inner sep=1pt] (taun11) at (4.6, 1) {$\tau_{n-1,1}$};
            \node (xn1) at (5.7, 0.7) {$x_{n-1}$};
            \node[inner sep=1pt] (taun1) at (6.7, 0.4) {$\tau_{n,1}$};
            \node (xn) at (7.5, 0) {$x_n$};
            
            \node (tau) at (-0.5, 0.3) {$\tau$};
            \node (tau00) at (1, -1.8) {$\tau_{00}$};
            \node (sig10) at (1.7, -1.65) {$\sigma_{10}$};
            \node (tau10) at (2.4, -1.5) {$\tau_{10}$};
            \node (taun10) at (4.5, -1.15) {$\tau_{n-1,0}$};
            \node (sign0) at (5.4, -0.9) {$\sigma_{n,0}$};
            \node[inner sep=1pt] (taun0) at (6.3, -0.7) {$\tau_{n,0}$};
            \node[red] (L1) at (-0.5, -0.7) {$L_1$};
            \draw[-{Stealth}, bend right] (tau01) edge (0.5, 0.75) (x0) edge (1, 0) (tau11) edge (1.6, 0.3) (x1) edge (2, 0) (taun11) edge (4, 0.3) (taun1) edge (6, 0.2) (xn1) edge (5, 0);
            \draw[-{Stealth}, bend left] (tau) edge (0, 0.7) (tau00) edge (0.5, -0.75) (sig10) edge (1, -0.5) (tau10) edge (1.6, -0.3) (taun10) edge (4, -0.3) (sign0) edge (4.9, -0.25) (taun0) edge (6, -0.2);
            \draw[-{Stealth}, bend left, red] (L1) edge (0.03, -0.3);
        \end{tikzpicture}
    \subcaption{For $m=1$, the simplicial complex $K'=K_{n,m}$ given in \textbf{Case 2}.}
    \label{fig:nonpoly_dgm_case2_simpcomp}
    \end{subfigure}
    
    \vspace{\baselineskip}
    
    \begin{subfigure}[h]{0.9\linewidth}
        \centering
        \begin{tikzpicture}[scale=0.6,
                            every pic/.style={scale=0.8,transform shape}]
            \foreach \x in {0, 1, 2, 3}{
                \foreach \y in {0, 1, 2, 3}{
                \tikzmath{ 
                    integer \xx, \yy;
                    \xx = \x-1; \yy = \y-1;
                if \y==0
                    then {
                        {\draw (3*\x, 3*\y) node[minimum size=1.5cm, draw, thick] (\x_\y) {};}; }
                    else { 
                        {\draw (3*\x, 3*\y) node[minimum size=1.5cm, draw, thick] (\x_\y) {};}; 
                        {\draw[-{Stealth}, thick] (\x_\yy) -- (\x_\y); };
                         };
                if \x!=0 then {
                    {\draw[-{Stealth}, thick] (\xx_\y) -- (\x_\y); }; 
                };
                if \y==0 then {
                    {\node (\x_num) at (3*\x, -1.8) {$\x$};};
                };
                if \x==0 then {
                    {\node (\y_num) at (-1.8, 3*\y) {$\y$};};
                };  
                if (\x+\y)>=3 then {
                    if (\x==3)&&(\y==3) then {
                        {\fill[gray!30] (8, 8) -- ++(0.5, 1) -- ++(-0.5, 1) -- ++(2, -1) -- cycle;};
                        {\draw[gray, very thick] (8.5, 9) -- ++(1.5, 0);};
                    };
                    if not((\x==0)&&(\y==3)) then {
                        {\pic[scale=0.75] at (\x_\y) {L0};};
                    };
                    if not((\x==1)&&(\y==2)) then {
                        {\pic[scale=0.75] at (\x_\y) {L1};};
                    };
                    if not((\x==2)&&(\y==1)) then {
                        {\pic[scale=0.75] at (\x_\y) {L2};};
                    };
                    if not((\x==3)&&(\y==0)) then {
                        {\pic[scale=0.75] at (\x_\y) {L3};};
                    };
                    {\draw[very thick, gray] (3*\x-1, 3*\y-1) -- ++(0, 2); };
                };  }
            } }
            \foreach \y in {0, 1, 2, 3}{
                \tikzmath{
                    int \yy; \yy=3-\y;
                }
                \pic[scale=0.75] at (-4, 3*\y) {L\yy};
                \node[color=Color\yy] at (-5.6, 3*\y) {$L_{\yy}$};
            }
        \end{tikzpicture}
    \subcaption{For $\m=1$ and $n=3$, the filtration $\Fcal'=\Fcal'_{n,m}$ given in \textbf{Case 2}.}
    \label{fig:nonpoly_dgm_case2_filt}
    \end{subfigure}
    \caption{}
\end{figure}

Let $N:=\Hrm_\m(\Fcal';\F):[n]^2\rightarrow \vect$. We claim that
$N$ is isomorphic to $M$, given in 
\Cref{eq:M}. Indeed,
\[N{(i, j)} \cong
\begin{cases}
0,&\mbox{if $i+j<n$}
\\ \langle [c_0], \cdots, [c_{i-1}], [c_{i+1}], \cdots ,[c_n] \rangle \cong \F^{n},&\mbox{if $i+j=n$}
\\ \langle [c_0], \cdots ,[c_n] \rangle \cong \F^{n+1},&\mbox{if $i+j>n$ and $(i, j)\neq (n, n)$}
\\ \F,&\mbox{if $(i,j)=(n,n)$}
\end{cases}\]
and for each pair $p \leq q$ in $[n]^2$, $N(p \leq q) \cong M(p \leq q)$.
Therefore, 
there is no $k\in \N$ such that $\norm{\dgm_\m(\Fcal')} = \norm{\dgm_N}=\norm{\dgm_M}=2^{n+1}-1\notin O(N^{k})$ for 
$N$, the number of simplices in the filtration $\Fcal'$, which belongs to $O(n^{m+2})$. 
\end{proof}

In order to prove Theorem \ref{thm:nonpolynomial size1} for the case where $d > 2$, we utilize a result from \cite[Section 3.2]{botnan2024signed} and techniques present in \cite[Section 3]{clause2022discriminating}. 
For any connected $P' \subset P$, let $\overline{P'}$ be the convex hull of $P'$, which is the smallest interval in $P$ containing $P'$.
By a \emph{morphism of lattices} $g : P \rightarrow Q$, we mean a monotone map between lattices $P$ and $Q$, such that for any $p_1, p_2 \in P$, $g(p_1 \wedge p_2) = g(p_1) \wedge g(p_2)$ and $g(p_1 \vee p_2) = g(p_1) \vee g(p_2)$.

\begin{proof}[Proof of Theorem \ref{thm:nonpolynomial size1} when $d>2$]
    Let $m=0$, and let $\Fcal: [n]^2 \rightarrow \Delta K$ be the filtration defined in Equation \eqref{eq:F}.
    Letting 
    $\pi : [n]^d \rightarrow [n]^2$ be the map
    $(x_1, x_2, \cdots, x_d) \mapsto (x_1, x_2)$,
    we obtain a simplicial filtration $\widetilde{\Fcal} := \Fcal \circ \pi : [n]^d \rightarrow \Delta K$, since $\pi$ is monotone.
    Let $M := \Hrm_\m(\Fcal;\F)$, $\Mtilde := \Hrm_\m(\widetilde{\Fcal};\F)$, and
    let $f: \Int([n]^d) \rightarrow \Z$ be defined as
    \begin{equation} \label{eq:f_is_dgm}
        f(I):=
        \begin{cases}
        \dgm_{M}(I'), &\mbox{ if $I = I' \times [n]^{d-2}$ for $I' \in \Int([n]^2)$,} \\
        0, &\mbox{otherwise.}
        \end{cases}
    \end{equation}
    Let $I \in \Int([n]^d)$. Since $\pi$ is a morphism of lattices, by \cite[Corollary 3.13]{botnan2024signed}, we have $\rk_{\Mtilde}(I) = \rk_M\left(\overline{\pi(I)} \right)$. For $I' \in \Int([n]^2)$, we have $\pi^{-1} (I') = I' \times [n]^{d-2} \in \Int([n]^d)$ 
    and thus $I' \times [n]^{d-2} \supseteq I$ iff $I' \supseteq \overline{\pi(I)}$. Now, by the definitions of $f$ and $\dgm_{M}$,
    we have:
    \[ \rk_{\Mtilde}(I) = \rk_M\left(\overline{\pi(I)}\right) =\! \sum_{\substack{J' \supseteq \overline{\pi(I)} \\ J' \in \Int([n]^2) }}\! \dgm_M(J') =\! \sum_{\substack{J' \times [n]^{d-2} \supseteq I \\ J' \in \Int([n]^2) }} \! \dgm_M(J') =\! \sum_{\substack{J \supseteq I \\ J \in \Int([n]^d)}} f(J). \]
    By the existence and uniqueness of $\dgm_{\Mtilde}$ (Theorem \ref{thm:GPD_existence} \cref{item:finite_GPD} and \cref{item:uniqueness}), we have $f = \dgm_{\Mtilde}$. This equality and Equation \eqref{eq:f_is_dgm} imply that $ \norm{\dgm_m(\widetilde{\Fcal})} = \norm{\dgm_{\Mtilde}} = \norm{\dgm_M}\not\in  O(N^k)$ for \emph{any} $k\in \Zplus$ where $N \in O(n^{m+2})=O(n^2)$.
    
    In the case of  $m \geq 1$, the proof is similar; we simply replace $\Fcal$ above by the filtration $\Fcal'$ defined in Equation \eqref{eq:F'}.
\end{proof}

\begin{remark} {
Notably, the supports of the GPDs we have considered contain many overlapping intervals with opposite-signed values. This makes it difficult to interpret the information the GPDs convey, unlike persistence diagrams for one-parameter persistence modules.}
\end{remark}

\section{Super-Polynomial Growth of the GPDs Induced by Finite Metric Spaces}\label{sec:super-metric-space}

In this section, 
we show that the sizes of the 1-st GPDs of the sublevel-Rips, sublevel-\v{C}ech, degree-Rips, and degree-\v{C}ech bifiltrations are \emph{not} bounded by \emph{any} polynomial in the number of points in the input metric space  (\Cref{thm:sub-Rips,thm:degree-Rips}).
We prove these results by constructing specific point sets in $\R^3$ that induce persistence modules isomorphic, or structurally similar to, those appearing in the proof of \Cref{thm:nonpolynomial size1}. Also, we make use of Rota's Galois connection.

A \textbf{Galois connection} between two posets $P$ and $Q$ is a pair of monotone maps $g_1:P\rightleftarrows Q:g_2$ satisfying $g_1(p)\leq q$ iff $p \leq g_2(q)$ for any $p \in P$ and $q \in Q$.
Let $f:P \to Q$ be any function. For any $h:P \to \Z$, the \textbf{pushforward} of $h$ along $f$ is the function $f_{\sharp}h: Q \to \Z$ given by $f_{\sharp}h(q) = \sum\limits_{p \in f^{-1}(q)}h(p).$
For any $\ell:Q \to \Z$, the \textbf{pullback} of $\ell$ along $f$ is the function $f^{\sharp}\ell:P \to \Z$ defined by $f^{\sharp}\ell(p)=(\ell\circ f)(p)$. 
For any poset $P$, let $\partial_P$ denote the M\"obius inversion operator over $P$. To say that the GPD of a $P$-module $M$ exists is equivalent to stating that $\partial_{\Int(P)} \rk_M$ is well-defined and equals $\dgm_M$.

\begin{lemma}[{\cite[Theorem 3.1]{gulen2022galois}}]\label{prop:galois_connection}
    Let $P$ be a finite poset, $Q$ be any poset, and $g_1:P\rightleftarrows Q:g_2$ be a Galois connection. Then,
    $
    \partial_Q \circ g^{\sharp}_2 = {g_1}_{\sharp} \circ \partial_P.
    $
\end{lemma}

\Cref{{prop:galois_connection}} is crucial for establishing the following lemma, whose second item is essential for the proofs of both \Cref{thm:sub-Rips} and \Cref{thm:degree-Rips}. 

\begin{lemma}\label{lem:Interval_projection}
    Let $P$ be a finite poset, $M$ be a $P$-module, and $I \subset P$ be an interval of $P$. Then, 
    \begin{romanenumerate}
        \item\label{item:Interval_projection_1} for every $J \in \Int(I)$, we have \[
        \dgm_{M|_I}(J)=\sum\limits_{\substack{J'\in \Int(P)\\  J\in \Pi(J' \cap \ I) }} \dgm_{M}(J'),
    \] 
    where $\Pi(J'\cap I)$ denotes the coarsest partition of $J \cap I$ into disjoint intervals of $I$.
    \item\label{item:Interval_projection_2} Furthermore, if $P$ and $I$ are finite 2-d grid, then $\norm{\dgm_M} \geq \norm{\dgm_{M|_I}}.$
    \end{romanenumerate}
\end{lemma}
\begin{proof} 

\cref{item:Interval_projection_1}
Let $p$ be an arbitrary element of $I$.
Consider the two posets $\Int(I)_p:= \{J \in \Int(I) \mid p \in J \}$ and $\Int(P)_p:= \{J \in \Int(P) \mid p \in J \}$, both of which are full subposets and upper sets of $\Int(P)$. Clearly, we have
    $\Int(I)_p \subset \Int(P)_p$. 
    Now, we define the following two maps: 
    \begin{itemize}
        \item For $ J\in \Int(P)_p$, there is the unique coarsest partition $\Pi(J\cap I) = \{J_1 , \ldots, J_k\}$ of $J \cap I$ into disjoint intervals of $I$.
        Since $p \in J \cap I$, there exists some $J_j \in \Pi(J\cap I)$ 
        such that $p \in J_j$. We define the map
        \[
        \pi : \Int(P)_p \to \Int(I)_p, \quad \pi(J) = J_j.
        \]
        \item Let $i : \Int(I)_p \to \Int(P)_p$ be the inclusion map. 
    \end{itemize}
    We claim that the pair $\pi:\Int(P)_p 
    \rightleftarrows \Int(I)_p:i$ is a Galois connection.
    Let $J' \in \Int(P)_p$ and $J \in \Int(I)_p$.
    First, if $\pi(J') \supseteq J$, then $J' \supseteq \pi(J') \supseteq J = i(J)$.
    Conversely, assume $J' \supseteq i(J) = J$. 
    We have $J' \cap I \supseteq J$ for $J' \supseteq J$ and $I \supseteq J$.
    As $J$ is connected, it must be entirely contained in one of the elements of $\Pi(J'\cap I)$. Since $p \in \pi(J') \in \Pi(J'\cap I)$, we conclude that $\pi(J') \supseteq J$.
 
  Note that if $J_1\in \Int(P)_p$ and $J_2\in \Int(P)$ with $J_2\supseteq J_1$, then $J_2 \in \Int(P)_p$. Likewise, if $J_1\in \Int(I)_p$ and $J_2\in \Int(I)$ with $J_2\supseteq J_1$, then $J_2\in \Int(I)_p$.
  These two facts and the uniqueness of the GPD (\Cref{thm:GPD_existence} \cref{item:uniqueness}) imply: 
    \begin{align}
    &\partial_{\Int(P)_p}(\rk_M|_{\Int(P)_p}) = \left(\partial_{\Int(P)}(\rk_M)\right)|_{\Int(P)_p}=\dgm_{M}|_{\Int(P)_p}, \label{eq:long_dgm_P} \\
    &
    \partial_{\Int(I)_p}({\rk_{M|_I}}|_{\Int(I)_p}) = \left(\partial_{\Int(I)}(\rk_{M|_I})\right)|_{\Int(I)_p}=\dgm_{{M|_I}}|_{\Int(I)_p}. \label{eq:long_dgm_I}
    \end{align}
    Also, \Cref{prop:galois_connection} implies that 
    \begin{equation}\label{eq:Galois}\partial_{\Int(I)_p} ( i^{\sharp}( \rk_{M}{|_{\Int(P)_p}})){}=\pi_{\sharp}(\partial_{\Int(P)_p}( \rk_{M}{|_{\Int(P)_p}})){}.
    \end{equation}
    Since 
     $i^{\sharp}( \rk_{M}{|_{\Int(P)_p}})=\rk_M|_{\Int(I)_p}=\rk_{M|_I}|_{\Int(I)_p}$, 
    by \Cref{eq:long_dgm_I}, the LHS of \Cref{eq:Galois} equals
    $
    \dgm_{{M|_I}}|_{\Int(I)_p}.
    $  
 By \Cref{eq:long_dgm_P}, the RHS of \Cref{eq:Galois} equals $\pi_{\sharp} ( \dgm_{M}{|_{\Int(P)_p}})$. 
    Therefore, \Cref{eq:Galois} reads as, for $J\in \Int(I)_p$,
    \[
    \dgm_{M|_I}(J)=\pi_{\sharp} ( \dgm_{M}{|_{\Int(P)_p}})(J) 
    =\sum\limits_{\substack{J'\in \Int(P)_p\\ \pi(J')=J}  } \dgm_{M}(J')
    =\sum\limits_{\substack{J'\in \Int(P)\\  J\in \Pi(J' \cap \ I) }} \dgm_{M}(J').
    \]
    Since $p\in I$ is arbitrary, the claim follows.

    \cref{item:Interval_projection_2}
    Let $P$ and $I$ be finite 2-d grids. Then, for every interval $J' \in \Int(P)$ with $J' \cap I \neq \emptyset$, the intersection $J' \cap I$ is itself an interval of $I$. 
    Therefore, we can rewrite the equation in \Cref{item:Interval_projection_1} as 
    \[
        \dgm_{M|_I}(J)=\sum\limits_{\substack{J'\in \Int(P)\\ J' \cap \ I = J}} \dgm_{M}(J').
    \]

    Hence, for each $J \in \Int(I)$ such that $\dgm_{M|_I}(J) \ne 0$, there exists $J' \in \Int(P)$ so that $J' \cap I = J$ and $\dgm_M(J') \ne 0$. This fact ensures the existence of a surjective map
    from the support of $\dgm_M$ to that of $\dgm_{M|_I}$, which implies $\norm{\dgm_M} \geq \norm{\dgm_{M|_I}}$. \qedhere
\end{proof}
\begin{remark}
    \Cref{lem:Interval_projection} \cref{item:Interval_projection_2} 
    does not generally 
    hold in higher dimensions, i.e. for $d \geq 3$, there exist $d$-d grids $P$ and $I$ with $I \subseteq P$ and a $P$-module $M$ such that $\norm{\dgm_M} < \norm{\dgm_{M|_I}}$.
\end{remark}

\subsection{Sublevel-Rips and Sublevel-\v{C}ech Bifiltrations}

 For $r \in \Rplus$, the \textbf{Vietoris-Rips complex} of a metric space $(X, \dd_X)$ at scale $r$ is defined as
\begin{align*}
        \Rips((X, \dd_X))_r &= \left\{\text{finite }\sigma \subset X : \max\limits_{x, y \in \sigma} \dd_X(x, y) \leq r \right\}.
\end{align*}
Now assume that $X \subset \R^d$, where $\R^d$ is equipped with a metric $\dd$, and $\dd_X = \dd|_{X \times X}$. The \textbf{\v{C}ech complex of $(X,d_X)$ in $(\R^d,\dd)$ at scale $r$} is
\begin{align*}
    \Cech((X, \dd_X))_r &= \left\{\text{finite }\sigma \subset X : \bigcap_{x \in \sigma} B_r(x)\neq \emptyset \right\},
\end{align*} 
where $B_r(x) = \{ y \in \R^d : \dd(x, y) \leq r \}$.

\begin{lemma}[{\cite[Lemma VII]{ghrist2005coverage}}]\label{lem:Cech_Rips}
    Let $(\R^d, \rho)$ be the Euclidean space with the supremum metric $\rho$. For any finite $X\subset \R^d$ and any $r \in \R_{\geq 0}$, we have
    $\Cech((X,\rho|_{X\times X}))_r = \Rips((X,\rho|_{X\times X}))_{2r}$.
\end{lemma}

    Let $\gamma: X \to \R$ be a function.
    For any $a\in \R$ and any $r\in \Rplus$ (see e.g. \cite{carlsson2009theory}), 
    \begin{romanenumerate}
    \item the \textbf{sublevel-Rips complex at scale $r$ and threshold $a$}  is $\Rips^{\downarrow}((X, \dd, \gamma))_{a,r} := $ \\ $ \Rips(\gamma^{-1}((-\infty, a]))_{r}$. The \textbf{sublevel-Rips bifiltration} of $(X, \dd, \gamma)$ is the simplicial filtration $\Rips^{\downarrow}((X, \dd, \gamma)) := \{\Rips^{\downarrow}((X, \dd, \gamma))_{a,r}\}_{a \in \R, r \in \R_{\geq 0}}$ over $\R \times \Rplus$. 

    \item the \textbf{sublevel-\v{C}ech complex at scale $r$ and threshold $a$}  is $\Cech^{\downarrow}((X, \dd, \gamma))_{a,r} :=$ \\ $ \Cech(\gamma^{-1}((-\infty, a]))_{r}$. The \textbf{sublevel-\v{C}ech bifiltration} of $(X, \dd, \gamma)$ is the simplicial filtration $\Cech^{\downarrow}((X, \dd, \gamma)) := \{\Cech^{\downarrow}((X, \dd, \gamma))_{a,r}\}_{a \in \R, r \in \R_{\geq 0}}$ over $\R \times \Rplus$.
    \end{romanenumerate}

Let
\begin{equation}\label{eq:gamma_and_T}
  \Gamma_X:=\im (\gamma)  \subset \R \ \  \mbox{and} \ \ T_X:=\im (\dd) \subset \Rplus.  
\end{equation}
By Remarks \ref{rem:finite_filtration_and_fp} and \ref{rem:discrete_poset}, for each $\m\in \Zplus$, 
the size of the GPD of $M:=\Hrm_\m(\Rips^{\downarrow}((X, \dd, \gamma)),\F):\R \times \Rplus \rightarrow \vect$ equals 
$\norm{\dgm_{M|_{\Gamma_X \times T_X}}}$.

\begin{theorem}[Super-polynomial GPD of Sublevel-Rips and Sublevel-\v{C}ech filtrations]
\label{thm:sub-Rips}

    There does \emph{not} exist $k\in \N$ such that for \emph{every} metric space $(X, \dd)$ and \emph{every} function $\gamma: X \to \R$, $\norm{\dgm_1(\Fcal_{(X,\dd, \gamma)})} $ 
    belongs to $ O(\abs{X}^k)$, where $\Fcal_{(X,\dd, \gamma)}\in \{\Rips^{\downarrow}((X,\dd, \gamma)),\Cech^{\downarrow}((X,\dd, \gamma))\}$.

\end{theorem}

Given any set $A\subset \R^d$ and $r\in \R$, let $rA:=\{ra\in \R^d:a\in A\}$. When $d\geq 2$, let $\pi:\R^d\rightarrow \R^2$ be the projection $(x_1,\ldots,x_d)\mapsto (x_1,x_2)$.

\begin{proof} 
\begin{figure}[!h]
    \centering
    \begin{tikzpicture}[scale=0.4, on grid]
    \foreach \y in {0,1,2,3,4,5,6} {
        \tikzmath{ integer \z, \p, \k, \kk;
            \z = 3*\y; \p = mod(\y,2); 
            \k = div(\y,2); \kk = (3-\k);
        }
        \begin{scope}[yshift=\y*4cm,
                      every node/.append style={yslant=0.5,xslant=-1.2, inner sep=2pt},
                      yslant=0.5,xslant=-1.2]
            \filldraw[fill=white, fill opacity=0.9, dashed] (-3, -3) rectangle (3, 3);
            \draw[-Stealth] (-3.5, 0) -- (3.5, 0); 
            \draw[-Stealth] (0, 3.5) -- (0, -3.5);
            \foreach \angle in {0, 90, 180, 270}{ 
                \tikzmath{
                    if \p==0 then {
                        for \i in {0,1,2,3}{
                            if \i != \k then {
                                {\fill[blue] (\angle:2.4-0.3*\i) node (\k_\i_\angle) {} circle (4pt);};
                            };
                        };
                    } else {
                        {\fill[red] (\angle:2.4) circle (5pt);};
                    };  }
            }
            \tikzmath{
                if \p==0 then {
                    {\draw[dashed, blue] (\k_\kk_270) -- ++(-2.4+0.3*\kk, 0);};
                    {\draw[Stealth-Stealth,blue] (\k_\kk_180) -- ++(0, -2.4+0.3*\kk);};
                };  }
        \end{scope}
        \node[scale=0.8, rotate=-18] at (6, 4*\y+1.3) {$z = \z$};
        \draw[|-Stealth] (9, 4*\y) node[below, text height=4mm] {0} -- (18, 4*\y) node[right] {$x, y$};
        \tikzmath{
            if \p==0 then {
                {\node[blue] at (-0.1*\k, 3.9*\y-0.5) {$b_\kk$};};
                {\node[blue] at (-8, 4*\y) {$Y_\k$};};
                for \i in {0,1,2,3}{
                    if \i != \k then { 
                        {\fill[blue] (16-0.8*\i, 4*\y) circle (5pt);};
                    };
                };
                {\node[blue] at (15, 4*\y-1) {$b_j$};};
            } else {
                {\node[red] at (-8, 4*\y) {$Z_\k$};};
                {\fill[red] (16, 4*\y) node[below right, text height=10pt] {$b_0$} circle (5pt);};
            };  }
    }
    \end{tikzpicture}
    \caption{For $n=3$, the set $X_n$ that is defined in the proof of Theorem \ref{thm:sub-Rips}.}
    \label{fig:sub_rips_3}
\end{figure}
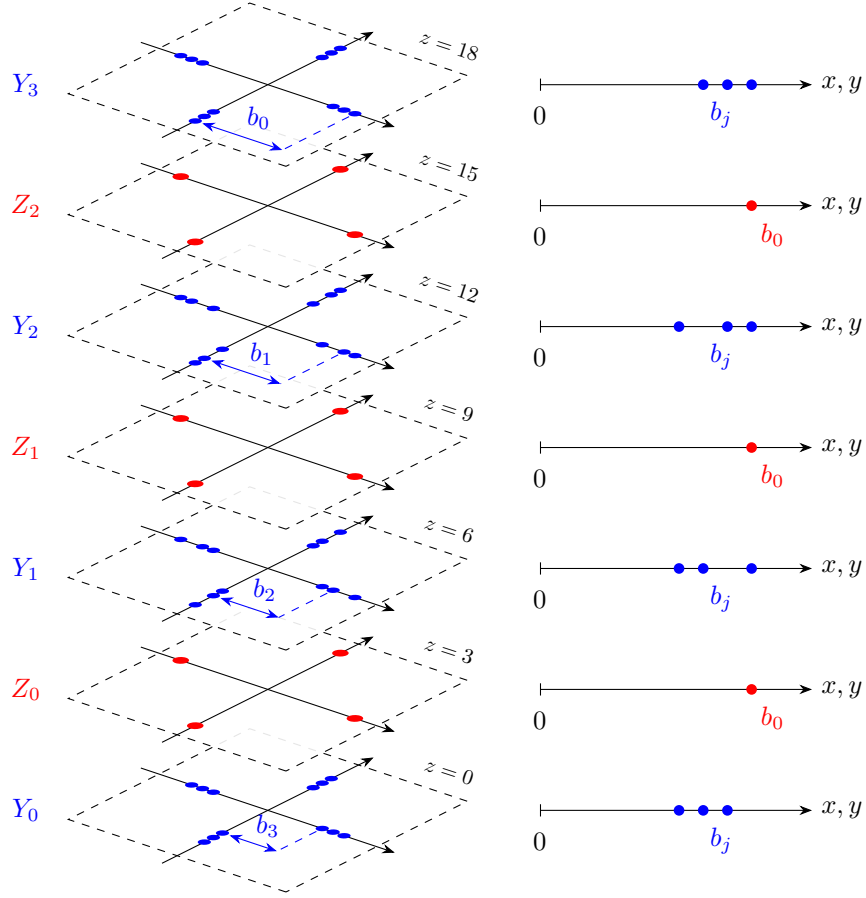
   
    By Lemma \ref{lem:Cech_Rips}, it suffices to show the statement when $\Fcal_{(X, \dd, \gamma)} = \Rips^{\downarrow}((X, \dd, \gamma))$.
    Let $n \geq 1$. We construct an $X_n \subset \R^3$, which consists of {$4n(n+2)$} points and inherits the metric $\dd_n$ from the supremum metric on $\R^3$. Our goal is to show that $\norm{\dgm_1(\Rips^{\downarrow}((X_n, \dd_n, \gamma_n)))}$ does not belong to $O(n^k)$ for any $k$, which then does not belong to $O(|X_n|^k)$ for any $k$.
    Let $\epsilon_n:=1/n^2$.
    Let 
    $X_n= \left( {\color{blue}\bigcup\limits_{i=0}^{n} Y_i} \right) \cup \left( {\color{red} \bigcup\limits_{i=0}^{n-1} Z_i} \right)$ (see Figure~\ref{fig:sub_rips_3}),
    where, for $i \in [n]$, 
    $Y_i$ is the subset of the plane $z=2in$ in $\R^3$ whose projected image $\pi Y_i$ onto $\R^2$ is described below, and $Z_i$ is the subset of the plane $z=(2i+1)n$ in $\R^3$ whose projected image $\pi Z_i$ onto $\R^2$ is described below.
    Let $A:=\{(1, 0), (0, 1), (-1, 0), (0, -1) \}$, $b_j:=n -j \epsilon_n$ for $j \in [n]$, $\pi Y_i :={\color{blue} \bigcup\limits_{\substack{0 \leq j \leq n \\ j\ne i}} b_j A }$, and $\pi Z_i := {\color{red} b_0A}$.
    Note that $\abs{X_n}=\left|{\color{blue}\bigcup\limits_{i=0}^{n} Y_i}\right|+\left|{\color{red} \bigcup\limits_{i=0}^{n-1} Z_i}\right|=4n(n+1)+4n=4n(n+2)$.
    Let $\gamma_n: X_n \to \R$ be
    \[
    \gamma_n(w) = 
    \begin{cases}
        j , & \text{if } w \in \bigcup\limits_{i=0}^{n} Y_i \text{ and } \pi(w) \in b_j A \text{ for some } 
        0 \leq j \leq n, \\
        n, & \text{if } w \in \bigcup\limits_{i=0}^{n-1} Z_i 
    \end{cases}
    \]
    so that $\Gamma_{X_n} = \im(\gamma_n) = \{0 < 1 <\ldots < n \} \subset \R$. 
    The \( 0 \)-simplices of \( X \), those that are subsets of \( Z \) are born at \( (n,0) \), while those that belong to \( b_j A \) under projection are born at \( (j,0) \).
    Now, consider the filtration $\Rips^{\downarrow}((X_n, \dd_n, \gamma_n))$ on $\Gamma_{X_n}\times T_{X_n}$. 
    The \( 0 \)-simplices that are subsets of \( \bigcup\limits_{i=0}^{n-1} Z_i \) are born at \( (n, 0) \), while those of $\bigcup\limits_{i=0}^{n} Y_i$ that project into \( b_j A \) are born at \( (j, 0) \).    

    For each $j\in [n]$, we have that $b_j\in T_{X_n}$ (which is clear from \Cref{fig:sub_rips_3}).
    In addition, we claim that the set
    $T'_n :=$ $\{ b_n < b_{n-1} < \ldots < b_0  \} 
    = \{n - n\epsilon_n < n - (n - 1)\epsilon_n < \ldots < n
    \}$
    is an interval of $T_{X_n}$.
    Indeed, the fact that $\epsilon_n$ equals $\min\limits_{\substack{x, y, z, w \in X_n \\ \dd_n(x,y) \neq \dd_n(z,w)}} | \dd_n(x, y) - \dd_n(z, w)|$ ensures that no element of $T_{X_n}$ lies between two consecutive elements of $T'_n$. Therefore, we can reindex $T'_n$ as $\{\beta_0< \beta_{1}< \ldots< \beta_{n}\}$. 
    Now, we consider the finite 2-d grid $\Gamma_{X_n} \times T'_n$ and its interval
    \begin{equation}\label{eq:Interval_sub_U}
        U :=\; \{(i,\beta_{j})\in \Gamma_{X_n} \times T'_n : \ i + j \geq n, 0 \leq j \leq n \}
    \end{equation}
    (see Figure \ref{fig:sub_rips_1}).
    By \Cref{lem:Interval_projection}\cref{item:Interval_projection_2} and the fact that $|X_n| = 4n(n+2)$, it suffices to show that the size of the 1-st GPD of $\Rips^{\downarrow}((X_n, \dd_n, \gamma_n))$ restricted to $\Gamma_{X_n} \times T'_n$ is super-polynomial in $n$.
    
    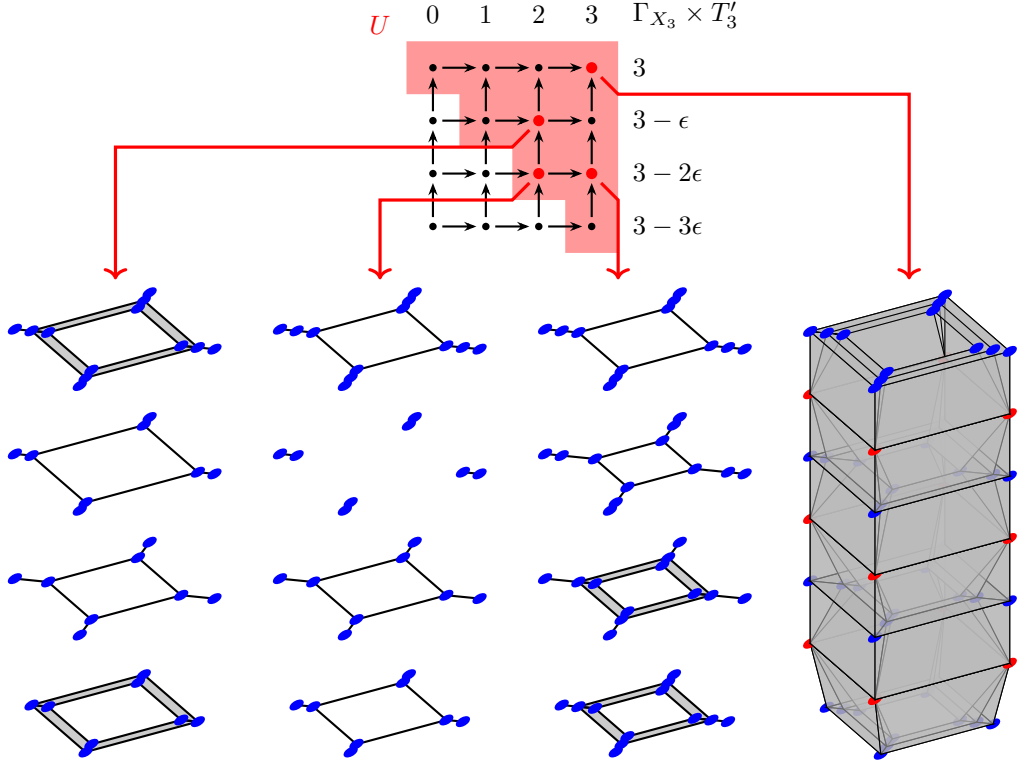
\begin{figure}[!h]
        \centering
        \begin{tikzpicture}[scale=0.7]
            \fill[fill=red!40] (2.5, -0.5) -- (2.5, 0.5) -- (1.5, 0.5) -- (1.5, 1.5) -- (0.5, 1.5) -- (0.5, 2.5) -- (-0.5, 2.5) -- (-0.5, 3.5) -- (3.5, 3.5) -- (3.5, -0.5) -- cycle;
            \foreach \x/\xx in {0, 1, 2, 3}{
                \foreach \y/\yy in {0, 1, 2, 3}{
                \tikzmath{ 
                    integer \xx, \yy;
                    \xx = \x-1; \yy = \y-1;
                if \y==0
                    then {
                        {\fill (\x, \y) node (\x_\y) {} circle (2pt);}; }
                    else { 
                        {\fill (\x, \y) node (\x_\y) {} circle (2pt);}; 
                        {\draw[-{Stealth[scale=0.7]}, thick] (\x_\yy) -- (\x_\y); };
                         };
                if \y==3 then {
                    {\node (\x_num) at (\x, \y+1) {$\x$};};
                };
                if \x!=0 then {
                    {\draw[-{Stealth[scale=0.7]}, thick] (\xx_\y) -- (\x_\y); }; 
                }; }
            } }
            \fill[red] (2_1) circle (3pt);
            \fill[red] (2_2) circle (3pt);
            \fill[red] (3_1) circle (3pt);
            \fill[red] (3_3) circle (3pt);
            \draw[red, very thick, ->] (2_1) -- (1.5, 0.5) -- (-1, 0.5) -- (-1, -1);
            \draw[red, very thick, ->] (2_2) -- (1.5, 1.5) -- (-6, 1.5) -- (-6, -1);
            \draw[red, very thick, ->] (3_1) -- (3.5, 0.5) -- (3.5, -1);
            \draw[red, very thick, ->] (3_3) -- (3.5, 2.5) -- (9, 2.5) -- (9, -1);
            \node[align=flush left, right] at (3.6, 0) {$3-3\epsilon$};
            \node[align=flush left, right] at (3.6, 1) {$3-2\epsilon$};
            \node[align=flush left, right] at (3.6, 2) {$3-\epsilon$};
            \node[align=flush left, right] at (3.6, 3) {3};
            \node[red] at (-1, 3.8) {$U$};
            \node[align=flush left, right] at (3.6, 4) {$\Gamma_{X_3} \times T'_3$};
        \end{tikzpicture}
        \\
        \begin{tikzpicture}[scale=0.55, y=0.5cm, on grid, baseline=0mm]
        \foreach \y in {0, 1, 2, 3}{
            \begin{scope}[  yshift=\y*3cm,
                            every node/.append style={yslant=-0.1,xslant=0.7},
                            yslant=-0.1,xslant=0.7]
                \foreach \angle in {0, 90, 180, 270}{ 
                    \foreach \i in {0,1,2}{
                        \tikzmath{
                            if \i != \y then {
                            {\node (\y_\i_\angle) at (\angle:2.4-0.4*\i) {};};
                            };  }  
                    }
                }
                \foreach \an/\gle in {0/90, 90/180, 180/270, 270/0}{
                    \tikzmath{ 
                        if \y==0 then {
                        {\filldraw[thick, fill=black, fill opacity=0.2] (\y_1_\an.center) -- (\y_1_\gle.center) -- (\y_2_\gle.center) -- (\y_2_\an.center)-- cycle;}; 
                        }; 
                        if \y==1 then {
                        {\draw[thick] (\y_0_\an.center) -- (\y_2_\an.center) -- (\y_2_\gle.center);};
                        }; 
                        if \y==2 then {
                        {\draw[thick] (\y_0_\an.center) -- (\y_1_\an.center) -- (\y_1_\gle.center);}; 
                        }; 
                        if \y==3 then {
                        {\draw[thick] (\y_0_\an.center) -- (\y_1_\an.center);};
                        {\filldraw[thick, fill=black, fill opacity=0.2] (\y_1_\an.center) -- (\y_1_\gle.center) -- (\y_2_\gle.center) -- (\y_2_\an.center)-- cycle;};
                        }; }
                }
                \foreach \angle in {0, 90, 180, 270}{
                    \tikzmath{
                          }
                    \foreach \i in {0,1,2}{
                        \tikzmath{
                            if \i != \y then {
                            {\fill[blue] (\y_\i_\angle) circle (4pt);};
                            };  }  
                    }
                }
            \end{scope}
        }
        \end{tikzpicture}
        \quad
        \begin{tikzpicture}[scale=0.55, y=0.5cm, on grid, baseline=0mm]
        \foreach \y in {0, 1, 2, 3}{
            \begin{scope}[  yshift=\y*3cm,
                            every node/.append style={yslant=-0.1,xslant=0.7},
                            yslant=-0.1,xslant=0.7]
                \foreach \angle in {0, 90, 180, 270}{ 
                    \foreach \i in {0,1,2}{
                        \tikzmath{
                            if \i != \y then {
                            {\node (\y_\i_\angle) at (\angle:2.4-0.4*\i) {};};
                            };  }  
                    }
                }
                \foreach \an/\gle in {0/90, 90/180, 180/270, 270/0}{
                    \tikzmath{ 
                        if \y!=2 then {
                        {\draw[thick] (\y_2_\an.center) -- (\y_2_\gle.center);}; 
                        }; }
                }
                \foreach \angle in {0, 90, 180, 270}{
                    \tikzmath{
                        if \y==0 then {
                        {\draw[thick] (\y_1_\angle.center) -- (\y_2_\angle.center);}; 
                        }; 
                        if \y==1 then {
                        {\draw[thick] (\y_0_\angle.center) -- (\y_2_\angle.center);}; 
                        }; 
                        if \y==2 then {
                        {\draw[thick] (\y_1_\angle.center) -- (\y_0_\angle.center);}; 
                        }; 
                        if \y==3 then {
                        {\draw[thick] (\y_0_\angle.center) -- (\y_2_\angle.center);}; 
                        };  }
                    \foreach \i in {0,1,2}{
                        \tikzmath{
                            if \i != \y then {
                            {\fill[blue] (\y_\i_\angle) circle (4pt);};
                            };  }  
                    }
                }
            \end{scope}
        }
        \end{tikzpicture}
        \quad
        \begin{tikzpicture}[scale=0.55, y=0.5cm, on grid, baseline=0mm]
        \foreach \y in {0, 1, 2, 3}{
            \begin{scope}[  yshift=\y*3cm,
                            every node/.append style={yslant=-0.1,xslant=0.7},
                            yslant=-0.1,xslant=0.7]
                \foreach \angle in {0, 90, 180, 270}{ 
                    \foreach \i in {0,1,2,3}{
                        \tikzmath{
                            if \i != \y then {
                            {\node (\y_\i_\angle) at (\angle:2.4-0.4*\i) {};};
                            };  }  
                    }
                }
                \foreach \an/\gle in {0/90, 90/180, 180/270, 270/0}{
                    \tikzmath{ 
                        if \y==0 then {
                        {\filldraw[thick, fill=black, fill opacity=0.2] (\y_3_\an.center) -- (\y_3_\gle.center) -- (\y_2_\gle.center) -- (\y_2_\an.center)-- cycle;}; 
                        {\draw[thick] (\y_1_\an.center) -- (\y_2_\an.center);};
                        }; 
                        if \y==1 then {
                        {\filldraw[thick, fill=black, fill opacity=0.2] (\y_3_\an.center) -- (\y_3_\gle.center) -- (\y_2_\gle.center) -- (\y_2_\an.center)-- cycle;}; 
                        {\draw[thick] (\y_0_\an.center) -- (\y_2_\an.center);};
                        }; 
                        if \y==2 then {
                        {\draw[thick] (\y_0_\an.center) -- (\y_3_\an.center) -- (\y_3_\gle.center);}; 
                        }; 
                        if \y==3 then {
                        {\draw[thick] (\y_0_\an.center) -- (\y_2_\an.center) -- (\y_2_\gle.center);};
                        }; }
                }
                \foreach \angle in {0, 90, 180, 270}{
                    \tikzmath{
                          }
                    \foreach \i in {0,1,2,3}{
                        \tikzmath{
                            if \i != \y then {
                            {\fill[blue] (\y_\i_\angle) circle (4pt);};
                            };  }  
                    }
                }
            \end{scope}
        }
        \end{tikzpicture}
        \quad
        \begin{tikzpicture}[scale=0.55, y=0.5cm, on grid, baseline=0.3mm,
                            facein/.style={fill=black!20, fill opacity=0.7},
                            facemid/.style={fill=black!20, fill opacity=0.5},
                            faceout/.style={fill=black!30, fill opacity=0.6}]
        \foreach \y in {0,1,...,6}{
                \tikzmath{ integer \p, \k;
                    \p = mod(\y,2); \k = div(\y,2); }
                \begin{scope}[  yshift=\y*1.5cm,
                                every node/.append style={yslant=-0.1,xslant=0.7},
                                yslant=-0.1,xslant=0.7]
                    \foreach \angle in {0, 90, 180, 270}{ 
                        \tikzmath{ 
                        if \p==0 then {
                            for \i in {0,1,2,3}{
                                if \i != \k then {
                                {\node (\k_\i_\angle) at (\angle:2.4-0.4*\i) {};};
                                };    
                            }; 
                        } else {
                            {\node (\k_R_\angle) at (\angle:2.4) {};};
                        }; }
                    }
                \end{scope}
            }
            \foreach \low/\high in {0/1, 1/2, 2/3}{
                \tikzmath{ 
                    for \angle in {0, 90, 180, 270}{
                        for \i in {0,1,2,3}{
                            if \i != \low then {
                            {\fill[blue,yslant=-0.1,xslant=0.7] (\low_\i_\angle) circle (4pt);};
                            };
                        };
                    };  
                    integer \exoL, \midL, \endoL, \exoH, \midH, \endoH;
                    \exoH = 0; \endoL = 3;
                    if \low==0 then {
                        \exoL = 1; \midH = 2;
                    } else {
                        \exoL = 0; \midH = 1;
                    };
                    if \high==3 then {
                        \midL = 1; \endoH = 2;
                    } else {
                        \midL = 2; \endoH = 3;
                    };
                }
                \foreach \an/\gle in {90/180, 0/90, 180/270, 270/0}{
                    \filldraw[faceout] 
                    (\low_\exoL_\an.center) -- (\low_\exoL_\gle.center) -- (\low_\endoL_\gle.center) -- (\low_\endoL_\an.center) -- cycle;
                    \draw (\low_\midL_\an.center) -- (\low_\midL_\gle.center);
                    \filldraw[facein] 
                    (\low_\endoL_\an.center) -- (\low_\endoL_\gle.center) -- (\low_R_\gle.center) -- (\low_R_\an.center) -- cycle;
                    \filldraw[facemid] 
                    (\low_\midL_\an.center) -- (\low_\midL_\gle.center) -- (\low_R_\gle.center) -- (\low_R_\an.center) -- cycle;
                    \filldraw[faceout] 
                    (\low_R_\an.center) -- (\low_R_\gle.center) -- (\low_\exoL_\gle.center) -- (\low_\exoL_\an.center) -- cycle;
                }
                \foreach \angle in {0, 90, 180, 270}{
                    \fill[red,yslant=-0.1,xslant=0.7] (\low_R_\angle) circle (4pt);
                }
                \foreach \an/\gle in {90/180, 0/90, 180/270, 270/0}{
                    \filldraw[facein] 
                    (\high_\endoH_\an.center) -- (\high_\endoH_\gle.center) -- (\low_R_\gle.center) -- (\low_R_\an.center) -- cycle;
                    \filldraw[facemid] 
                    (\high_\midH_\an.center) -- (\high_\midH_\gle.center) -- (\low_R_\gle.center) -- (\low_R_\an.center) -- cycle;
                    \filldraw[faceout] 
                    (\high_\exoH_\an.center) -- (\high_\exoH_\gle.center) -- (\low_R_\gle.center) -- (\low_R_\an.center) -- cycle;
                }
                \tikzmath{ 
                    if \high==3 then {
                    {\filldraw[faceout] 
                    (3_0_180.center) -- (3_0_90.center) -- (3_0_0.center) -- (3_0_270.center) -- (3_0_180.center) -- (3_2_180.center) -- (3_2_270.center) -- (3_2_0.center) -- (3_2_90.center) -- (3_2_180.center) -- cycle;};
                    {\draw (3_1_180.center) -- (3_1_90.center) -- (3_1_0.center) -- (3_1_270.center) -- cycle;};
                    for \angle in {0, 90, 180, 270}{
                        for \i in {0,1,2}{
                            {\fill[blue,yslant=-0.1,xslant=0.7] (\high_\i_\angle) circle (4pt);};    
                        };
                    };  };  }
            }
        \end{tikzpicture}
        \caption{For $n=3$, the interval $U$ of the grid $\Gamma_{X_3} \times T'_3$ that is defined in Equation \eqref{eq:Interval_sub_U} is shown at the top. 
        The 1-st homologies of $\Rips^{\downarrow}((X_3, \dd_3, \gamma_3))_{2,3-\epsilon_3}$, $\Rips^{\downarrow}((X_3, \dd_3, \gamma_3))_{2,3-2\epsilon_3}$ and $\Rips^{\downarrow}((X_3, \dd_3, \gamma_3))_{3,3-2\epsilon_3}$ are either 3- or 4-dimensional.  The 1-cycles in these simplicial complexes are homologous within $\Rips^{\downarrow}((X_3, \dd_3, \gamma_3))_{3,3}$, whose underlying space is homotopy equivalent to the cylinder $S^1 \times [0,1]$.}
        \label{fig:sub_rips_1}
    \end{figure}
    Then, on the interval $\Gamma_{X_n} \times T'_n$, ${M_n}$ $ := \Hrm_1(\Rips^{\downarrow}((X_n, \dd_n, \gamma_n)); \F)$ is given by (see \Cref{fig:sub_rips_1}):
    \[
    {M_n}{(i, \beta_{j})}=
    \begin{cases}
        0,&\mbox{if $i+j<n$}
        \\\F^n,&\mbox{if $i+j=n$}
        \\\F^{n+1},&\mbox{if $i+j>n$ and $(i, j)\neq (n, n)$}
        \\\F,&\mbox{if $(i,j)=(n,n)$.}
    \end{cases}
    \]
    Let $i\in [n]$. Then, for every $(x,y)\in U$ that covers $(i, \beta_{n-i})$, we have that
    \begin{align}
    M_n((i, \beta_{n-i}) \leq (x, y))&&=&&\iota_i^n:&&\F^{n}&\to \F^{n+1}\nonumber
    \\
    &&&&&&(v_1, \dots, v_n) &\mapsto (v_1, \dots, v_{i}, 0, v_{i+1}, \dots, v_n)
    \end{align}
    The poset $\Gamma_{X_n} \times T'_n$ is isomorphic to the poset $[n]^2$ 
    via the isomorphism $(i, \beta_{j+k}) \sim (i, j)$. The persistence module $M_n|_{\Gamma_{X_n} \times T'_n}$ is also isomorphic to the persistence module $M$ given in Equations \eqref{eq:M} and \eqref{eq:iota}. It is proven that $\norm{\dgm_M} \notin O(n^k)$ for any $k \in \Zplus$ in the proof of \Cref{thm:nonpolynomial size1}, which completes our proof.

    \end{proof}

\subsection{Degree-Rips and Degree-\v{C}ech Bifiltrations}

    Let $(X, \dd)$ be a metric space.
    For $d\in \Z$ and $r\in \Rplus$, 
    \begin{romanenumerate}
        \item  the \textbf{Degree-Rips complex at scale $r$ and degree $d$} \cite{lesnick2015interactive}, denoted $ \DRips((X, \dd))_{d,r} $, is the maximal subcomplex of $ \Rips((X, \dd))_{r} $ whose vertices have degree at least \( d - 1 \) in the 1-skeleton of $ \Rips((X, \dd))_{r} $.  The \textbf{Degree-Rips bifiltration} of $(X, \dd)$ is the simplicial filtration  $\DRips((X, \dd)):=\{\DRips((X, \dd))_{d, r} \}_{d\in \Zop, r\in \Rplus}$ over $\Zop \times \Rplus$. 
        \item The \textbf{Degree-\v{C}ech complex at scale $r$ and degree $d$} \cite{edelsbrunner2021multi}, denoted $ \DCech((X, \dd))_{d,r} $, is the maximal subcomplex of $ \Cech((X, \dd))_{r} $ whose vertices have degree at least \( d - 1 \) in the 1-skeleton of $ \Cech((X, \dd))_{r} $. The \textbf{Degree-\v{C}ech bifiltration} of $(X, \dd)$ is the simplicial filtration $\DCech((X, \dd)) := \{\DCech((X, \dd))_{d, r}\}_{d \in \Zop, r \in \R_{\geq 0}}$ over $\Zop \times \Rplus$.
    \end{romanenumerate}

We consider the finite subposets
\begin{equation*}
  J_X:= \{|X|-1< |X|-2 < \ldots < 1 \} \subset \Zop\ \  \mbox{and} \ \ T_X:=\im (\dd) \subset \Rplus.   
\end{equation*}
By Remarks \ref{rem:finite_filtration_and_fp} and \ref{rem:discrete_poset}, for each $\m\in \Zplus$, 
the size of the GPD of $M:=\Hrm_\m(\DRips((X, \dd)),\F):\Zop \times \Rplus \rightarrow \vect$ equals $\norm{\dgm_{M|_{J_X\times T_X}}}$.

\begin{theorem}[Super-polynomial GPD of Degree-Rips and Degree-\v{C}ech filtrations]\label{thm:degree-Rips}

 There does \emph{not} exist $k\in \N$ such that for \emph{every} finite metric space $(X,\dd)$, $\norm{\dgm_1(\Fcal_{(X,\dd)})} $ 
 belongs to $ O(\abs{X}^k)$, where $\Fcal_{(X,\dd)}\in \{\DRips((X,\dd)),\DCech((X,\dd))\}$.
\end{theorem}

\begin{proof}
    By Lemma \ref{lem:Cech_Rips}, it suffices to show the statement when $\Fcal_{(X, \dd)} = \DRips((X, \dd))$.

    Let $n\geq 1$. We construct an $X_n \subset \R^3$, which consists of $4(n+1)^2$ points and inherits the metric $\dd_n$ from the supremum metric on $\R^3$.
    Our goal is to show that $\norm{\dgm_1(\DRips((X_n, \dd_n)))}$ does not belong to $O(n^k)$ for any $k$, which then does not belong to $O(|X_n|^k)$ for any $k$.
      Let $\epsilon_n:=1/n^2$.
      Let $X_n:=\bigcup\limits_{i=0}^n Y_i $ (see Figure~\ref{fig:deg_rips_3}), where, for $i \in [n]$, $Y_i$ is the subset of the plane $z=(n+\epsilon_n)i$ in $\R^3$ whose projected image $\pi Y_i$ onto $\R^2$ is described below.
    Let $A:=\{(1, 0), (0, 1), (-1, 0), (0, -1) \}$, $B:=\{(1, \epsilon_n), (-\epsilon_n, 1), (-1, -\epsilon_n), (\epsilon_n, -1) \}$, and $a := n-n\epsilon_n$. For $j\in [n+1]$, let $b_j := a+j\epsilon_n = n+ (j-n)\epsilon_n$.
    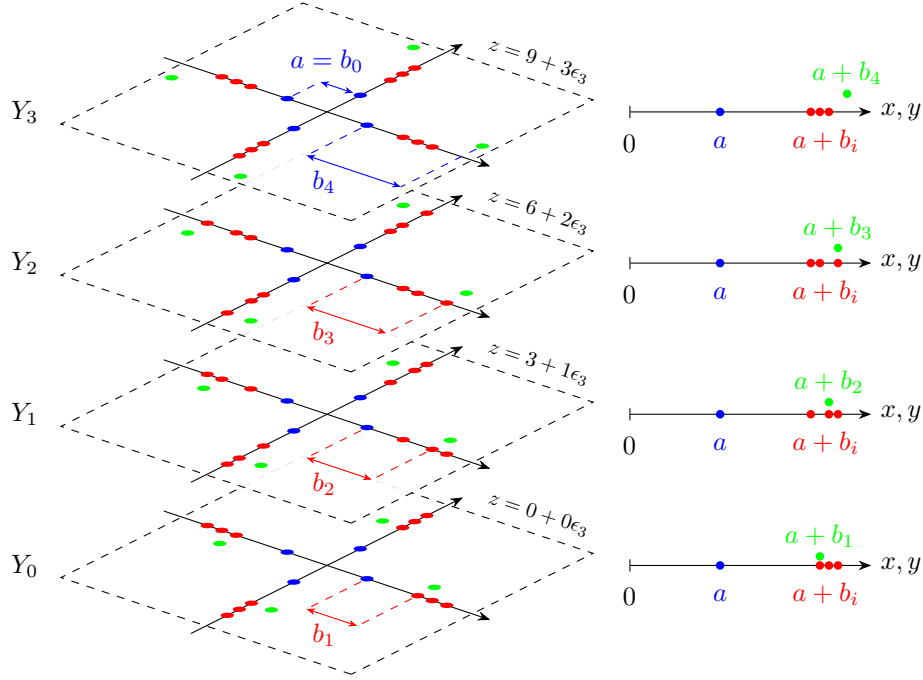
\begin{figure}
    \centering
    \begin{tikzpicture}[scale=0.4, on grid]
    \foreach \y/\yy/\z in {0/1/0, 1/2/3, 2/3/6, 3/3/9}{
        \begin{scope}[  yshift=\y*5cm,
                        every node/.append style={yslant=0.5,xslant=-1.2, inner sep=2pt},
                        yslant=0.5,xslant=-1.2]
            \filldraw[fill=white, fill opacity=0.9, dashed] (-4, -4) rectangle (4, 4);
            \draw[-Stealth] (-4.5, 0) -- (4.5, 0); 
            \draw[-Stealth] (0, 4.5) -- (0, -4.5);
            \foreach \angle in {0, 90, 180, 270}{ 
                \fill[blue] (\angle:1.1) node (\y_\angle) {} circle (4pt);
                \foreach \i in {0, 1, 2, 3}{
                    \tikzmath{
                    if \i==\y then {
                        {\fill[green] (\angle:2.5+0.4*\i)+(\angle+90:0.55+0.05*\i) node (\y_\angle_\i) {} circle (4pt);};}
                    else {{\fill[red] (\angle:2.1+0.4*\i) node (\y_\angle_\i) {} circle (4pt);};}; } 
                }
            } 
            \tikzmath{ if \y==3 
            then {
                {\draw[dashed, blue] (\y_90) -- ++(1.1, 0);};
                {\draw[{Stealth[scale=0.7]}-{Stealth[scale=0.7]}, blue] (\y_0) -- ++(0, 1.1);}; 
                {\draw[dashed, blue] (\y_270) -- ++(-2, 0);};
                {\draw[dashed, blue] (\y_270_\yy) -- ++(-2.7, 0);};
                {\draw[{Stealth[scale=0.7]}-{Stealth[scale=0.7]}, blue] (\y_270)++(-2, 0) -- ++(0, -1.4-0.4*\y);};
                }
            else {
                {\draw[dashed, red] (\y_270) -- ++(-2, 0);};
                {\draw[dashed, red] (\y_270_\yy) -- ++(-2, 0);};
                {\draw[{Stealth[scale=0.7]}-{Stealth[scale=0.7]}, red] (\y_270)++(-2, 0) -- ++(0, -1.4-0.4*\y);};
                }; 
            }
        \end{scope}
        \node at (-10, 5*\y) {$Y_\y$};
        \node[rotate=-18, scale=0.8] at (7, 5*\y+1.7) {$z = \z + \y \epsilon_3$};
        \draw[|-Stealth] (10, 5*\y) node[below, text height=4mm] {0} -- (18, 5*\y) node[right] {$x, y$};
        \fill[blue] (13, 5*\y) node[below, text height=10pt] {$a$} circle (4pt);
        \node[red, below, text height=10pt] at (16.5, 5*\y) {$a + b_i$};
        \foreach \i/\j in {0/1, 1/2, 2/3, 3/4}{
            \tikzmath{
            if \i==\y then {
            {\fill[green] (16.3+0.3*\i, 5*\y+0.3+0.1*\i) node[above, text height=10pt] {$a + b_\j$} circle (4pt);};}
            else {{\fill[red] (16+0.3*\i, 5*\y) circle (4pt);};}; }
        }
        \tikzmath{ if \y!=3 then {{\node[red] at (-0.1, -2.3+5*\y) {$b_\yy$};};}; }
    }
    \node[blue] at (0, 16.7) {$a=b_0$};
    \node[blue] at (-0.1, 12.7) {$b_4$};
    \end{tikzpicture}
    \caption{For $n=3$, the set $X_n$ that is defined in the proof of Theorem \ref{thm:degree-Rips} }
    \label{fig:deg_rips_3}
    \end{figure}
    Let 
    \begin{align*}
         &\pi Y_i := {\color{blue} a A} \cup \Big({\color{red} \bigcup\limits_{\substack{0 \leq j \leq n \\ j\ne i}} (a + b_i)A } \Big) \cup {\color{green} (a+b_{i+1})B}.
    \end{align*}

\begin{claim}
    The two sets $J'_n :=$ $\{n+4 < n+3 < \ldots < 3 \}\ \subset J_{X_n}$ and $T'_n := $ $\{b_0 < b_1 < \ldots < b_n < b_{n+1} \}\subset T_{X_n}$ are intervals of $J_{X_n}$ and $T_{X_n}$, respectively.
\end{claim}
   Assuming that Claim 1 holds, the set $J'_n \times T'_n$ is a finite 2-d grid, which is an interval of $J_{X_n} \times T_{X_n}$, and the set $U$ defined below is an interval of $J'_n \times T'_n$ (see Figure \ref{fig:deg_rips_1}) 
    \begin{equation}\label{eq:Interval_deg_U}
        U:=\; \{(\alpha_{i}, b_{j})\in J'_n \times T'_n : 
         i + j \geq n,\, 0 \leq i \leq n+1,\, 0 \leq j \leq n+1 \},
    \end{equation}
    where $\alpha_i$ denote the $i$-th element of $J'_n$ (with 0-based indexing).
    We now proceed with the proof of Claim 1.
    \begin{claimproof}
    Clearly, the set $J'_n$ is an interval of $J_{X_n}$. It remains to prove that the set $T'_n$ is an interval of $T_{X_n}$.
    For $i\in[n+1]$, we have $b_i \in T_{X_n}$ (which is clear from \Cref{fig:deg_rips_3}).
    In addition, the fact that $\epsilon_n$ equals $\min\limits_{\substack{x, y, z, w \in X_n \\ \dd_n(x,y) \neq \dd_n(z,w)}} | \dd_n(x, y) - \dd_n(z, w)|$ ensures that no element of $T_{X_n}$ lies between two consecutive elements of $T'_n$. 
    \end{claimproof}
    By \Cref{lem:Interval_projection}\cref{item:Interval_projection_2}, it suffices to show that the size of the 1-st GPD of $\DRips((X_n, \dd_n))$ restricted to $J'_n \times T'_n$ is super-polynomial.

    \begin{figure}[!ht]
        \centering
        \begin{tikzpicture}[every node/.style={inner sep=3pt}, scale=0.55]
            \fill[fill=cyan!40] (2.5, -0.5) -- (2.5, 0.5) -- (1.5, 0.5) -- (1.5, 1.5) -- (0.5, 1.5) -- (0.5, 2.5) -- (-0.5, 2.5) -- (-0.5, 3.5) -- (1.5, 3.5) -- (1.5, 2.5) -- (2.5, 2.5) -- (2.5, 1.5) -- (3.5, 1.5) -- (3.5, -0.5) -- cycle;
            \foreach \x in {0,1, ..., 4}{
                \tikzmath{ int \xop, \xx; \xop = 7-\x; \xx = \x-1;}
                \foreach \y in {0,1, ..., 4}{
                \tikzmath{ int \yy; \yy = \y-1;
                    {\fill (\x, \y) node (\x_\y) {} circle (2pt);}; 
                    if \x!=0 then {
                        {\draw[-{Stealth[scale=0.7]}, thick] (\xx_\y) -- (\x_\y); };
                    };
                    if \y!=0 then {
                        {\draw[-{Stealth[scale=0.7]}, thick] (\x_\yy) -- (\x_\y); };
                    };
                    if \y==4 then {
                        {\node (\xop) at (\x, \y+1) {$\xop$}; };    
                    };     
                    }  
                }
            }
            \node[right] at (4.5, 0) {$3-3\epsilon$};
            \node[right] at (4.5, 1) {$3-2\epsilon$};
            \node[right] at (4.5, 2) {$3-\epsilon$};
            \node[right] at (4.5, 3) {$3$};
            \node[right] at (4.5, 4) {$3+\epsilon$};
            \fill[cyan] (2_2) circle (3pt);
            \fill[cyan] (2_1) circle (3pt);
            \fill[cyan] (3_1) circle (3pt);
            \draw[cyan, very thick, ->] (2_2) -- (1.5, 1.5) -- (-3, 1.5) -- (-3, -1);
            \draw[cyan, very thick, ->] (2_1) -- (2.5, 0.5) -- (2.5, -1);
            \draw[cyan, very thick, ->] (3_1) -- (3.5, 0.5) -- (8, 0.5) -- (8, -1);
            \node[right] at (4.5, 5) {$J'_3 \times T'_3$};
            \node[cyan] at (-1, 2.5) {$D$};
        \end{tikzpicture}
        \begin{tikzpicture}[every node/.style={inner sep=3pt}, scale=0.55]
            \fill[fill=red!40] (2.5, -0.5) -- (2.5, 0.5) -- (1.5, 0.5) -- (1.5, 1.5) -- (0.5, 1.5) -- (0.5, 2.5) -- (-0.5, 2.5) -- (-0.5, 4.5) -- (4.5, 4.5) -- (4.5, -0.5) -- cycle;
            \foreach \x in {0,1, ..., 4}{
                \tikzmath{ int \xop, \xx; \xop = 7-\x; \xx = \x-1;}
                \foreach \y in {0,1, ..., 4}{
                \tikzmath{ int \yy; \yy = \y-1;
                    {\fill (\x, \y) node (\x_\y) {} circle (2pt);}; 
                    if \x!=0 then {
                        {\draw[-{Stealth[scale=0.7]}, thick] (\xx_\y) -- (\x_\y); };
                    };
                    if \y!=0 then {
                        {\draw[-{Stealth[scale=0.7]}, thick] (\x_\yy) -- (\x_\y); };
                    };
                    if \y==4 then {
                        {\node at (\x, \y+1) {$\xop$}; };
                    };
                    }  
                }
            }
            \node[right] at (4.5, 0) {$3-3\epsilon$};
            \node[right] at (4.5, 1) {$3-2\epsilon$};
            \node[right] at (4.5, 2) {$3-\epsilon$};
            \node[right] at (4.5, 3) {$3$};
            \node[right] at (4.5, 4) {$3+\epsilon$};
            \node[right] at (4.5, 5) {$J'_3 \times T'_3$};
            \node[red] at (-1, 2.5) {$U$};
            \fill[red] (4_4) circle (3pt);
            \draw[red, very thick, ->] (4_4) -- (3.5, 3.5) -- (3.5, -1);
        \end{tikzpicture}
        \\
        \begin{tikzpicture}[scale=0.45, y=0.5cm, on grid, baseline=0mm]
        \foreach \y in {0, 1, 2, 3}{
            \begin{scope}[  yshift=\y*3cm,
                            every circle node/.append style={yslant=-0.1,xslant=0.7},
                            yslant=-0.1,xslant=0.7]
                \foreach \angle in {0, 90, 180, 270}{ 
                    \fill[blue] (\angle:1.1) node (B_\y_\angle) {} circle (4pt);
                    \foreach \i in {0, 1, 2}{
                        \tikzmath{ if \i==\y 
                        then {
                            if \y!=2 then {
                            {\fill[green] (\angle:2.5+0.4*\i)+(\angle+90:0.55+0.05*\y) node (RG_\y_\angle_\i) {} circle (4pt);}; }; }
                        else {
                            {\fill[red] (\angle:2.1+0.4*\i) node (RG_\y_\angle_\i) {} circle (4pt);};}; }
                    }
                    \tikzmath{ 
                    if \y<2 then {
                        {\filldraw[thick, fill=black, fill opacity=0.2] (RG_\y_\angle_\y.center) (RG_\y_\angle_1.center) -- (RG_\y_\angle_0.center) -- (B_\y_\angle.center) -- cycle;};
                        {\filldraw[thick, fill=black, fill opacity=0.2] (RG_\y_\angle_0.center) -- (RG_\y_\angle_1.center) -- (RG_\y_\angle_2.center) -- cycle;}; };
                    if \y==2 then {
                        {\draw[thick] (RG_\y_\angle_1.center) -- (B_\y_\angle.center);}; };
                    if \y==3 then {
                        {\draw[thick] (RG_\y_\angle_2.center) -- (B_\y_\angle.center);}; }; }
                }
                \draw[thick] (B_\y_0.center) -- (B_\y_90.center) -- (B_\y_180.center) -- (B_\y_270.center) -- cycle;
                \draw[very thick, blue, -{Stealth}] (0:0.8) -- (90:0.8) -- (180:0.8) -- (270:0.8) -- (0:0.8);
                \node[blue] at (0.8, -1.8) {$\delta_\y$};
            \end{scope}
        }
        \end{tikzpicture}
        \begin{tikzpicture}[scale=0.45, y=0.5cm, on grid, baseline=0mm]
        \foreach \y in {0, 1, 2, 3}{
            \begin{scope}[  yshift=\y*3cm,
                            every node/.append style={yslant=-0.1,xslant=0.7},
                            yslant=-0.1,xslant=0.7]
                \foreach \angle in {0, 90, 180, 270}{ 
                    \tikzmath{ if \y!=1 then {
                    {\fill[blue] (\angle:1.1) node (B_\y_\angle) {} circle (4pt);}; }; }
                    \foreach \i in {0, 1}{
                        \tikzmath{ if \i==\y 
                        then {
                            if \y!=1 then {
                            {\fill[green] (\angle:2.5+0.4*\i)+(\angle+90:0.55+0.05*\y) node (RG_\y_\angle_\i) {} circle (4pt);}; }; }
                        else {
                            {\fill[red] (\angle:2.1+0.4*\i)  node (RG_\y_\angle_\i) {} circle (4pt);};}; }
                    }
                    \tikzmath{ 
                    if \y==0 then {
                        print{\filldraw[thick, fill=black, fill opacity=0.2] (RG_\y_\angle_0.center) -- (RG_\y_\angle_1.center) -- (B_\y_\angle.center) -- cycle;};};
                    if \y>1 then {
                        print{\draw[thick] (RG_\y_\angle_1.center) -- (B_\y_\angle.center);};}; }
                }
                \tikzmath{ if \y!=1 then {
                {\draw[thick] (B_\y_0.center) -- (B_\y_90.center) -- (B_\y_180.center) -- (B_\y_270.center) -- cycle;}; }; }
            \end{scope}
        }
        \end{tikzpicture}
        \begin{tikzpicture}[scale=0.45, y=0.5cm, on grid, baseline=0mm]
        \foreach \y in {0, 1, 2, 3}{
            \begin{scope}[  yshift=\y*3cm,
                            every node/.append style={yslant=-0.1,xslant=0.7},
                            yslant=-0.1,xslant=0.7]
                \foreach \angle in {0, 90, 180, 270}{ 
                    \fill[blue] (\angle:1.1) node (B_\y_\angle) {} circle (4pt);  
                    \foreach \i in {0, 1}{
                        \tikzmath{ if \i==\y 
                        then {
                            if \y!=1 then {
                            {\fill[green] (\angle:2.5+0.4*\i)+(\angle+90:0.55+0.05*\y) node (RG_\y_\angle_\i) {} circle (4pt);}; }; }
                        else {
                            {\fill[red] (\angle:2.1+0.4*\i)  node (RG_\y_\angle_\i) {} circle (4pt);};}; }
                    }
                    \tikzmath{ 
                    if \y==0 then {
                        {\filldraw[thick, fill=black, fill opacity=0.2] (RG_\y_\angle_0.center) -- (RG_\y_\angle_1.center) -- (B_\y_\angle.center) -- cycle;};};
                    if \y>1 then {
                        {\draw[thick] (RG_\y_\angle_1.center) -- (B_\y_\angle.center);};};
                    if \y==1 then {
                        {\draw[thick] (RG_\y_\angle_0.center) -- (B_\y_\angle.center);};};
                    }
                }
                \draw[thick] (B_\y_0.center) -- (B_\y_90.center) -- (B_\y_180.center) -- (B_\y_270.center) -- cycle;
            \end{scope}
        }
        \end{tikzpicture}
        \begin{tikzpicture}[scale=0.45, y=0.5cm, on grid, baseline=0.3mm,
                            face/.style={fill=black!30, fill opacity=0.8}]
        \foreach \y in {0, 1, 2, 3}{
            \begin{scope}[  yshift=\y*3cm,
                            every node/.append style={yslant=-0.1,xslant=0.7},
                            yslant=-0.1,xslant=0.7]
                \foreach \angle in {0, 90, 180, 270}{ 
                    \node (B_\y_\angle) at (\angle:1.1) {};
                    \foreach \i in {0, 1, 2, 3}{
                        \tikzmath{ if \i==\y 
                        then {
                            {\draw (\angle:2.5+0.4*\i)+(\angle+90:0.55+0.05*\y) node (RG_\y_\angle_\i) {};};}
                        else {
                            {\node (RG_\y_\angle_\i) at (\angle:2.1+0.4*\i) {}; };}; }
                    }
                }
            \end{scope}
        }
        \foreach \low/\high in {0/1, 1/2, 2/3}{
        \tikzmath{ if \low==0 
        then {
            for \angle in {0, 90, 180, 270}{
                {\fill[blue,yslant=-0.1,xslant=0.7] (B_\low_\angle) circle (4pt);};
                {\fill[green,yslant=-0.1,xslant=0.7] (RG_\low_\angle_0) circle (4pt);};
                {\fill[red,yslant=-0.1,xslant=0.7] (RG_\low_\angle_1) circle (4pt);};
                {\fill[red,yslant=-0.1,xslant=0.7] (RG_\low_\angle_2) circle (4pt);};
                {\fill[red,yslant=-0.1,xslant=0.7] (RG_\low_\angle_3) circle (4pt);};
                {\filldraw[face] (B_\low_\angle.center) -- (RG_\low_\angle_0.center) -- (RG_\low_\angle_3.center) -- cycle;};
                {\draw[thick] (RG_\low_\angle_1.center) -- (RG_\low_\angle_0.center) -- (RG_\low_\angle_2.center);};
            };
            }; }
            \foreach \angle/\next in {90/180, 0/90, 180/270, 270/0}{
                \filldraw[face] 
                (B_\high_\angle.center) -- (B_\high_\next.center) -- (B_\low_\next.center) -- (B_\low_\angle.center) -- cycle;
                \filldraw[face] 
                (B_\high_\angle.center) -- (RG_\high_\angle_3.center) -- (RG_\low_\angle_3.center) -- (B_\low_\angle.center) -- cycle;
                \filldraw[face] 
                (B_\high_\angle.center) -- (RG_\high_\angle_2.center) -- (RG_\low_\angle_2.center) -- (B_\low_\angle.center) -- cycle;
                \filldraw[face] 
                (B_\high_\angle.center) -- (RG_\high_\angle_1.center) -- (RG_\low_\angle_1.center) -- (B_\low_\angle.center) -- cycle;
                \filldraw[face] 
                (B_\high_\angle.center) -- (RG_\high_\angle_0.center) -- (RG_\low_\angle_0.center) -- (B_\low_\angle.center) -- cycle;
                \foreach \i/\j in {1/2, 2/3, 3/1}{
                \foreach \k in {0, 1, 2, 3}{
                    \filldraw[fill=black!30, fill opacity=0.2] 
                    (RG_\high_\angle_\i.center) -- (RG_\high_\angle_\j.center) -- (RG_\low_\angle_\k.center);
                    \filldraw[fill=black!20, fill opacity=0.1] 
                    (RG_\low_\angle_\i.center) -- (RG_\low_\angle_\j.center) -- (RG_\high_\angle_\k.center);
                }}
                \fill[blue,yslant=-0.1,xslant=0.7] (B_\high_\angle) circle (4pt);
                \fill[red,yslant=-0.1,xslant=0.7] (RG_\high_\angle_0) circle (4pt);
                \fill[red,yslant=-0.1,xslant=0.7] (RG_\high_\angle_1) circle (4pt);
                \fill[red,yslant=-0.1,xslant=0.7] (RG_\high_\angle_2) circle (4pt);
                \fill[red,yslant=-0.1,xslant=0.7] (RG_\high_\angle_3) circle (4pt);
                \fill[green,yslant=-0.1,xslant=0.7] (RG_\high_\angle_\high) circle (4pt);
            }
        }
        \end{tikzpicture}
        \caption{For $n=3$, the intervals $U$ and $D$ of $J'_3\times T'_3$ that are defined in Equations \eqref{eq:Interval_deg_U} and \eqref{eq:Interval_deg_D} are shown at the top. The 1-st homologies of
        $\DRips((X_3, \dd_3))_{5, 3-\epsilon_3}$, $\DRips((X_3, \dd_3))_{5, 3-2\epsilon_3}$ and $\DRips((X_3, \dd_3))_{4, 3-2\epsilon_3}$ are either 3- or 4-dimensional. The body of $\DRips((X_3, \dd_3))_{3, 3+\epsilon_3}$ is homotopy equivalent to the cylinder $S^1 \times [0,1]$, and thus its 1-st homology is 1-dimensional.}
        \label{fig:deg_rips_1}
    \end{figure}
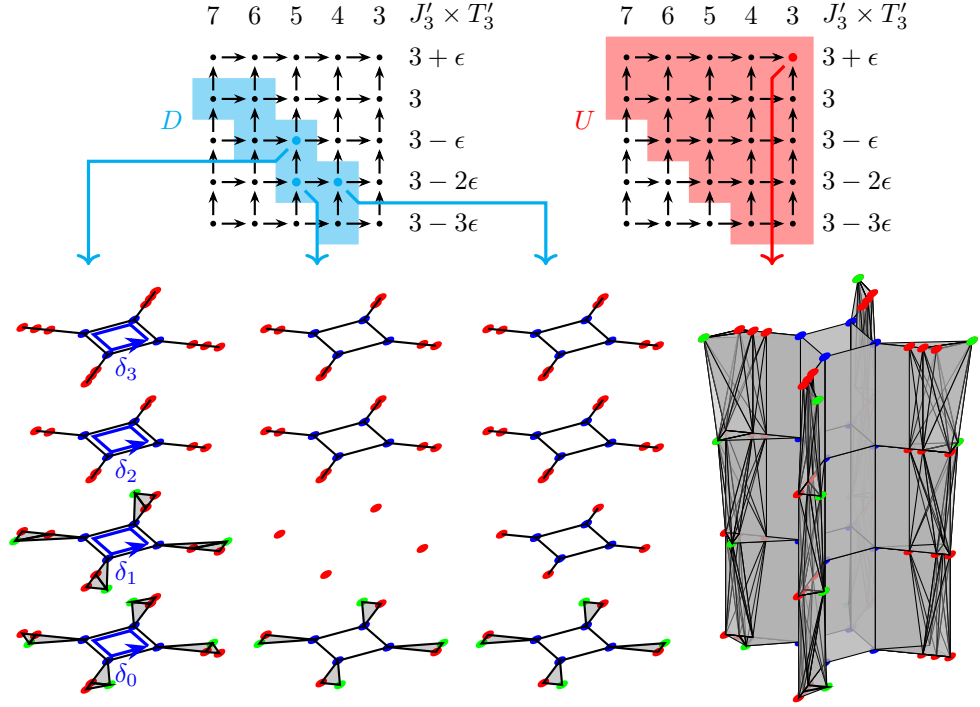
    Then, on the interval $J'_n \times T'_n$, $M_n$ $:=\Hrm_1(\DRips((X_n, \dd_n)),\F)$ is given by (see \Cref{fig:deg_rips_1}):
    \[
    M_n{(\alpha_{i }, b_{j})}=
    \begin{cases}
        0,&\mbox{if $i+j<n$} \\
        \F^n,&\mbox{if $i+j=n$} \\
        \F^{n+1},&\mbox{if $i+j>n$ and $j < n + 1$} \\
        \F,&\mbox{if $j = n + 1$.}
    \end{cases}
    \]
    Let $i\in [n]$. Then, for every $(x,y)\in U$ that covers $(\alpha_{i }, b_{n-i})$, we have that
    \begin{align}
    M_n((\alpha_{i}, b_{n-i}) \leq (x, y))&&=&&\iota_i^n:&&\F^{n}&\to \F^{n+1}\nonumber
    \\
    &&&&&&(v_1, \dots, v_n) &\mapsto (v_1, \dots, v_{i}, 0, v_{i+1}, \dots, v_n) \label{eq:iota_deg}
    \end{align}
    We consider the following interval of $J'_n \times T'_n$ (see Figure \ref{fig:deg_rips_1}): 
    \begin{equation}\label{eq:Interval_deg_D}
        D:=\{(\alpha_{i },b_{j})\in J'_n \times T'_n : \ n \leq i + j\leq  n + 1, \ i\leq n, \ j\leq n \}.
    \end{equation}

\begin{claim} 
For $J\in \Int(J'_n \times T'_n)$ with $J \not\subseteq U$ or $J = U$, $\rk_{M_n}(J)=0$.
\end{claim}
\begin{proof}
First, assume that 
there exists a point $p \in J-U$. Then, by the construction of $M_n$, $\dim (M_n(p)) = 0$ and thus $\rk_{M_n} (J)=0$ by monotonicity of $\rk_{M_n}$ (Remark \ref{rem:basic_properties}~\cref{item:monotonicity}).
Next, assume that $J=U$.
Since $\rk_{M_n}(U)=\rank(\varprojlim {M_n}|_U\rightarrow \varinjlim {M_n}|_U)$, it suffices to show $\varprojlim {M_n}|_U=0$. Again, since $D=\minzz(U)$
, by Lemma \ref{lem:fence}, it suffices to show that $\varprojlim {M_n}|_D=0$. 
    The persistence module $\varprojlim {M_n}|_D$ is isomorphic to the persistence module $\varprojlim {M}|_D$ defined in the proof of \Cref{thm:nonpolynomial size1}. Therefore, the remainder of the proof follows similarly to the proof of \Cref{cl:rank_J=0} in \Cref{thm:nonpolynomial size1}. 
\claimqedhere
\end{proof}
    Now, for each $i\in [n]$, let $U_i := U-\{(\alpha_{n-i },b_{i})\}$
    , which is an interval of $J'_n \times T'_n$.
\begin{claim} $\rk_{M_n}(U_i)=1$ for every $i\in [n]$. \end{claim}
\begin{proof}
    Fix $i\in [n]$. For every $p\in U_i$, the $1$-cycle (see Figure \ref{fig:deg_rips_1})
    \[ 
    \delta_i = \sum_{ x=0 }^{3}{ \Big[\big(a \cos\frac{\pi x}{2}, a \sin\frac{\pi x}{2}, i(n+\epsilon_n) \big), \big(a \cos\frac{\pi (x+1)}{2}, a \sin\frac{\pi (x+1)}{2}, i(n+\epsilon_n) \big) \Big]} 
    \]    
    is carried by the simplicial complex $\DRips(X_n)_{p}$. Hence, according to Convention \ref{con:limit_formula} and the definition of ${M_n}$, we have that $( [\delta_i] )_{p\in U_i}\in \varprojlim {M_n}|_{U_i} \subset \bigoplus_{p\in U_i} \Hrm_1(\DRips(X_n)_p;\F)$. Since $\max(U_i)=\{(\alpha_{n+1},b_{n+1})\}$, by Lemma \ref{lem:fence}, this tuple $( [\delta_i] )_{p \in U_i}$ maps to \\ $[\delta_i] \in \Hrm_1(\DRips(X_n)_{(\alpha_{n + 1},b_{n+1})};\F) = {M_n}(\alpha_{n + 1},b_{n+1})$ via the limit-to-colimit map of ${M_n}$ over $U_i$, followed by the isomorphism $\varinjlim {M_n}|_{U_i} \cong {M_n}(\alpha_{n + 1},b_{n+1})$. \\ Since $[\delta_i]\in \Hrm_1(\DRips(X_n)_{(\alpha_{n +1},b_{n+1})};\F)\cong \F$ is nonzero, we have $\rk_{M_n}(U_i)= 1$. \claimqedhere
\end{proof}
    The proofs of \Cref{cl:dgm_Ui=1_new} and  \Cref{cl:dgm_wedgeS_new} 
    below can be obtained by simply replacing $M$ in the proofs of \Cref{cl:dgm_Ui=1} and \Cref{cl:dgm_wedgeS} by $M_n$.

\begin{claim} \label{cl:dgm_Ui=1_new} $\dgm_{M_n}(U_i)=1$ for every $i\in [n]$.
\end{claim}

    For each nonempty $S \subset [n]$, we define $U_{S} := \bigcap_{i \in S} U_i$, which is also an interval of $J'_n \times T'_n$. 
    
\begin{claim} \label{cl:dgm_wedgeS_new}
    For each nonempty $S \subset [n]$, we have that $\dgm_{M_n}(U_{S}) = (-1)^{|S|+1}$. 
\end{claim}
    
    The number of nonempty subsets $S\subset [n]$ is $2^{n+1}-1$. Hence, $\norm{\dgm_{M_n}}$
    is at least $2^{n+1}-1$, which does not belong to $O(|X_n|^k)$ for any $k\in \Z_{\geq 0}$. 
 \end{proof}

It is known that, for $m\in \N$, adding $m$ integers {has} time complexity $O(m)$ \cite{jevrabek2023note}. Hence, as a direct corollary of \Cref{thm:nonpolynomial size1,thm:degree-Rips,thm:sub-Rips}, we obtain:

\begin{corollary}\label{cor:exp-time} 
Let $n,d\in \Zplus$, $d\geq 2$, and $I\in \Int([n]^d)$. For an arbitrary $[n]^d$-module $M$,
computing $\dgm_M(I)$ via the M\"obius inversion of the {GRI} 
of $M$ requires at least $\Theta(2^n)$ in time. In particular, $M$ can arise as the homology of a simplicial filtration over $[n]^d$, with the number of simplices being polynomial in $n$.
\end{corollary}
\begin{proof}
Let $M$ be the persistence module defined in the proof of Theorem \ref{thm:nonpolynomial size1} for $d=2$ and $m=0$. Let $I:=U_{[n]}$,
 as described immediately above Claim \ref{cl:dgm_wedgeS} in that proof. By Lemma \ref{lem:J-I}, we have 
$\dgm_M(I)=\sum_{S'\subsetneq [n]}(-1)^{\abs{S'}}=(-1)^{n+2}$.  
 This sum contains $2^{n+1}-1$ nonzero summands, 
requiring $\Theta(2^n)$ time for computation.
\end{proof}

\section{Conclusion}\label{sec:discussion}

We showed that the size of the GPD can be super-polynomial in the number of simplices in a given multi-parameter filtration, and such sizes can arise even from well-known constructions. Furthermore, we noted that the GRI can be ``locally densely supported'' (see Claim \ref{cl:dgm_wedgeS}), which suggests that using the M\"obius inversion of the GRI to compute the GPD can be intractable. These findings highlight the need to compute the GPD (or any invariant approximating it) directly, without relying on the GRI.

From a different perspective, developing an output-sensitive algorithm for computing the GPD might be a more practical approach, potentially extending or adapting ideas from \cite{morozov2021output}. {Moreover, combining optimization techniques to address challenges in computing the GPD appears to be a promising research direction. In line with this, \cite{carriere2024sparsification} proposes a method for ``sparsifying'' 
the GPD via gradient descent.}

It also remains an interesting question whether special types of multi-parameter filtrations not explored in this work can also give rise to GPDs of super-polynomial size. Examples include 
multicover and barycentric (subdivision) bifiltrations 
\cite{sheehy2012multicover}.

\bibliographystyle{plainurl}
\bibliography{bib.bib}

\end{document}